\documentclass[11pt]{amsart}
\usepackage{amsmath}
\usepackage{amssymb}
\usepackage{mathrsfs}
\usepackage{esint}
\usepackage{tikz}
\usepackage{comment}
\usepackage{ulem}
\usepackage{xcolor}
\usepackage{hyperref}

\theoremstyle{plain}
\newtheorem{theorem}{Theorem}[section]
\newtheorem{lemma}[theorem]{Lemma}
\newtheorem{proposition}[theorem]{Proposition}

\theoremstyle{definition}

\newtheorem{definition}[theorem]{Definition}

\numberwithin{equation}{section}

\makeatletter
\@namedef{subjclassname@2020}{%
  \textup{2020} Mathematics Subject Classification}
\makeatother

\def\be{\begin{equation}}
\def\ee{\end{equation}}

\def\R{\mathbb R}
\def\N{\mathbb N}

\def\Z{\mathbb Z}

\def\epsilon{\varepsilon}

\def\({\left(}
\def\){\right)}

\begin{document}

\title[Solutions to the Liouville Equation]
{A Parametrization of Solutions\\ to the Liouville Equation in a Punctured Disc}

\author[Han]{Qing Han}
\address{Department of Mathematics\\
University of Notre Dame\\
Notre Dame, IN 46556, USA} 
\email{qhan@nd.edu}
\author[Hung]{Jui-Yun Hung}
\address{Mathematical Sciences Institute\\
The Australian National University\\
Canberra, ACT 2600, Australia} 
\email{JUI-YUN.HUNG@anu.edu.au}


\begin{abstract}
In this paper, we study parametrization of solutions to the Liouville equation near their isolated singular points. 
We demonstrate that prescribing asymptotic behaviors forms a new type of problem
leading to all such solutions. The harmonic functions determining the asymptotic behaviors
play the role of ``free data" as the boundary values in the boundary value problems.
\end{abstract}

\thanks{Q. Han acknowledges the support of NSF Grant DMS-2305038.}


\maketitle

\section{Introduction}\label{sec-Intro}

In this paper, we study the Liouville equation
\begin{align}
\label{eq-LiouvilleEquation}
\Delta u + e^u = 0 \quad\text{in }B_1 \setminus \{0\},
\end{align}
where $B_1 \subseteq\R^2$ is the unit ball in $\R^2$. This is the 2-dimensional analogue of the Yamabe equation
\begin{align}
\label{eq-Yamabe-equation}
\Delta u + \frac{1}{4} n (n-2) u^{\frac{n+2}{n-2}} = 0 \text{ in }B_1 \setminus \{0\},
\end{align}
for $n\ge 3$. In a pioneering paper \cite{Caffarelli1989AsymptoticSA}, 
Caffarelli, Gidas, and Spruck proved that any positive singular solutions 
of \eqref{eq-Yamabe-equation} in $B_1 \setminus \{0\}$ 
are asymptotic to radial singular solutions 
of \eqref{eq-Yamabe-equation} in $\R^n \setminus\{0\}$. 
Subsequently, in \cite{Korevaar1998RefinedAF}, Korevaar, Mazzeo, Pacard, and Schoen 
studied refined asymptotics and expanded solutions to the next order. 
More recently in \cite{Han2019AsymptoticEO}, 
Han, X. Li, and Y. Li studied expansions of positive singular solutions of \eqref{eq-Yamabe-equation} 
up to arbitrary orders. 

The equation \eqref{eq-LiouvilleEquation} is the special form of the Gauss curvature equation and 
has been extensively studied from different aspects. 
For example, Malchiodi \cite{Malchioldi2014, Malchioldi2020} studied the Liouville equation 
from a variational point of view. 
Refer to \cite{Bogatov-Kichenassamy2021} for a historical account and \cite{Cai&Lai2024} 
for a survey of classification results for solutions to the Liouville equation. 
In this paper, we study solutions of \eqref{eq-LiouvilleEquation} 
with an isolated singular point at the origin.


Concerning the asymptotic expansions of solutions to \eqref{eq-LiouvilleEquation}, 
Chou and Wan \cite{ChouWan1994} derived the leading order term 
and proved the following result.

\smallskip 
\noindent 
{\bf Theorem A.}
{\it Let $u$ be a smooth solution of \eqref{eq-LiouvilleEquation} in $B_1\setminus\{0\}$. 
Then, there exists a constant $\nu_0 > -2$ such that, for any $x \in B_{1/2}\setminus \{0\}$,
$$\big|u(x) - \nu_0\ln |x|\big| <C,$$
for some constant $C>0$, if and only if
\begin{align}
\label{eq-finite-exp-integral-condition}
\int_{B_\epsilon\setminus\{0\}} e^u dx < \infty,
\end{align}
for some $\epsilon > 0$.}
\smallskip 

Condition \eqref{eq-finite-exp-integral-condition} is essential for solutions 
of \eqref{eq-LiouvilleEquation} to have asymptotic radial symmetry.

Recently, Wan, Guo, and Yang \cite{GuoWanYang2021CalVar} 
established the expansions up to arbitrary orders and proved the following result.

\smallskip 
\noindent 
{\bf Theorem B.}
{\it Let $u$ be a smooth solution of \eqref{eq-LiouvilleEquation} in $B_1\setminus\{0\}$ 
satisfying \eqref{eq-finite-exp-integral-condition}. 
Then, there exist constants $\nu_0> -2$, $a \in \R$, a strictly increasing sequence $\{\mu_i\}_{i\ge 1}$, divergent to $\infty$,
and a sequence of smooth functions $\{c_i\}_{i\ge 1}$ on $\mathbb S^1$,
such that, for any positive integer $\ell$ and any $x\in B_{1/2}\setminus\{0\}$, 
\begin{align}\label{eq-expansions-arbitrary-order}
&\Big|u(x)-\nu_0\ln|x|-a -\sum_{i=1}^\ell 
c_{i}(\theta)|x|^{\mu_i}\Big|\le C|x|^{\mu_{\ell+1}},
\end{align}
where $C$ is a positive constant and $\theta = {x}/{|x|}$. 
}
\smallskip 

We point out that Theorem B is the unified version of the main result in \cite{GuoWanYang2021CalVar}. 
There are eight cases in Theorem 1.3 \cite{GuoWanYang2021CalVar}, 
depending on the range of the constant $\nu_0$. 

The estimate \eqref{eq-expansions-arbitrary-order} demonstrates a decay pattern of 
$u(x)-\nu_0\ln|x|-a$. 
The collection of decay orders $\{\mu_i\}_{i\ge 1}$ is determined by the constant $\nu_0$, 
which depends on the specific solution $u$ of \eqref{eq-LiouvilleEquation} in $B_1\setminus\{0\}$. 
The constant $a$ and the collection of coefficients $\{c_i(\theta)\}_{i\ge 1}$ of course are determined by $u$. 

To motivate our results, we examine the structure of the coefficients $\{c_{i}\}_{i\ge 1}$ 
in \eqref{eq-expansions-arbitrary-order} closely. 
In this paper, we first provide an alternative derivation of \eqref{eq-expansions-arbitrary-order}.  
Our proof of \eqref{eq-expansions-arbitrary-order} demonstrates 
that the collection of coefficients $\{c_i\}_{i\ge 1}$ is determined by the constants $\nu_0$ and $a$, up to 
a sequence of spherical harmonics. 
In fact, $\{c_i\}_{i\ge 1}$ satisfies a family of recursive equations
and is determined uniquely by  $\nu_0$, $a$,  and a sequence of spherical harmonics $\{X_i\}_{i\ge 1}$ with $\deg X_i=i$.
We can form an infinite series 
\begin{align}\label{eq-intro-infinite-series}\sum_{\ell=1}^\infty c_{\ell}(\theta)|x|^{\mu_\ell}.\end{align}
We say that {\it the asymptotics of $u(x)-\nu_0\ln|x|-a$ is determined by $\{X_i\}_{i\ge 1}$} if 
$u$ satisfies \eqref{eq-expansions-arbitrary-order} for $\{c_{i}\}_{i\ge 1}$ as determined above and any $\ell$. 
Refer to Section \ref{sec-Formal-Expansions} for details in the setting of cylindrical coordinates. 

Theorem B simply asserts that for any solution $u$ of \eqref{eq-LiouvilleEquation} in $B_1\setminus\{0\}$
satisfying \eqref{eq-finite-exp-integral-condition}, 
the asymptotics of $u(x)-\nu_0\ln|x|-a$ is determined by some collection of 
spherical harmonics $\{X_i\}_{i\ge 1}$ with $\deg X_i=i$. 

In this paper, we study the converse and aim to construct solutions of \eqref{eq-LiouvilleEquation} 
near the origin with prescribed asymptotic behaviors at the origin.   
The following theorem is our main result.

\begin{theorem}
\label{thrm-correspondence-harmonic-to-solutions-Intro}
Let $\nu_0>-2$ and $a \in \R$ be constants, and $H$ be a harmonic function near $0$ with $H(0) = 0$. 
Express $H(x)=\sum_{i\ge 1} |x|^iX_i(\theta)$ for spherical harmonics $X_i$ with $\deg X_i=i$. 
Then, 
there exists a unique smooth solution $u$ of 
\eqref{eq-LiouvilleEquation} in a punctured ball near $0$ 
such that
\begin{itemize}
    \item[(1)] $\big|u(x) - \nu_0 \ln |x| - a \big| \le C |x|^\gamma$ for some constants $C>0$ and $\gamma > 0$;
    \item[(2)] the asymptotics of $u(x) - \nu_0\ln |x| - a$ is determined by $\{X_i\}_{i\ge 1}$.
\end{itemize}
\end{theorem}

Conclusion (1) asserts that $\nu_0 \ln |x| + a$ provides an initial approximation of $u(x)$ 
and (2) asserts that the infinite series in \eqref{eq-intro-infinite-series} 
describes completely the asymptotic behavior of $u(x) - \nu_0 \ln |x| - a$. 
In fact, we will prove that the infinite series in \eqref{eq-intro-infinite-series} 
converges uniformly and absolutely in a small punctured ball centered at the origin. 
Refer to Theorem \ref{thrm-main-theorem-cylindrical-coords}
for a more detailed version for \eqref{eq-LiouvilleEquation} in cylindrical coordinates.

We now briefly discuss how to prove
Theorem \ref{thrm-correspondence-harmonic-to-solutions-Intro}, or more specifically, 
the convergence of the infinite series in \eqref{eq-intro-infinite-series}. 
The basic idea is as follows. 
With the given sequence of spherical harmonics $\{X_i\}_{i\ge 1}$ 
in Theorem \ref{thrm-correspondence-harmonic-to-solutions-Intro}, each $c_\ell$ is determined by prior 
$c_1, \cdots, c_{\ell-1}$. 
We will derive a recursive estimate of $c_\ell$ 
in terms of $c_1, \cdots, c_{\ell-1}$ in an appropriate Sobolev norm and then investigate whether such 
recursive estimates permit us to prove the convergence of the series in \eqref{eq-intro-infinite-series}. 
In our formulation of Theorem \ref{thrm-correspondence-harmonic-to-solutions-Intro}, 
the collection of the spherical harmonics $\{X_i\}_{i\ge 1}$ is considered to be ``free data". 
However, the dependence of $\{c_\ell\}_{\ell\ge 1}$ on $\{X_i\}_{i\ge 1}$ is not clear. 
Some $c_\ell$ is given by a single $X_i$, some $c_\ell$ is given by 
the unique solution of some recursive equation, while other $c_\ell$ is $X_i$ 
corrected by such a unique solution. 
The lack of the clear dependence makes it difficult to derive effective estimates 
leading to the convergence of \eqref{eq-intro-infinite-series}. 
Refer to Proposition \ref{prop-polynomial-coefficients} for details. 

For a remedy, we employ  
multi-indices with countably many components to reformulate the series 
\eqref{eq-intro-infinite-series}. 
Specifically, for $i\ge 0$, denote by $e_i \in \Z_+^\infty$ the integer sequence 
whose $i$-th component is $1$ and all the other components are $0$, and then define
$$\varLambda = \Z_+\text{-}\text{span}\{ e_i;\,i\geq 0\} \setminus\{0\}.$$
In other words, $\varLambda$ consists of all non-zero sequences $\alpha = \{\alpha_i\}_{i\geq 0}$ 
such that $\alpha_i \in \Z_+$ for any $i\geq 0$ and $\alpha_i = 0$ for all but finitely many $i$. 
We emphasize that there are only finitely many nonzero components in each $\alpha \in \varLambda$. 
This is the reason we introduce a new notation $\varLambda$ instead of using $\Z_+^{\infty}\setminus\{0\}$.
For any $\alpha \in \varLambda$, we write
\begin{align*}
|\alpha| = \sum_{i=0}^\infty \alpha_i.
\end{align*}
This is a finite sum since $\alpha_i = 0$ for all but finitely many $i$.

With such multi-indices, we define $P_{e_0} = -e^a/(\nu_0+2)^2$  
and $P_{e_i} = X_i(\theta)$ for $i\geq 1$ for the given 
sequence of spherical harmonics $\{X_i\}_{i\ge 1}$, considered as given data.  
We reformulate the series in \eqref{eq-intro-infinite-series} as 
\begin{align}\label{eq-intro-infinite-series-reformulated-multi}
\sum_{\alpha\in \varLambda} c_\alpha(\theta)|x|^{\mu_\alpha},\end{align}
where $\{\mu_\alpha\}_{\alpha\in \varLambda}$ is the index set $\{\mu_\ell\}_{\ell\ge 1}$ 
reparametrized by $\varLambda$ and each $c_\alpha$ is determined by 
prior $c_\beta$ with $|\beta|<|\alpha|$. 
We need to point out there are infinitely many $\beta\in \varLambda$ with $|\beta|=j$ for any fixed $j\ge 1$. 
However, each $c_\alpha$ in \eqref{eq-intro-infinite-series-reformulated-multi} 
is determined by only finitely many prior $c_\beta$'s. 
Refer to Proposition \ref{prop-formal-expansion-multi-index} for details. 

At the first glance, it seems much more complicated to parametrize the series 
by infinite dimensional multi-indices as in \eqref{eq-intro-infinite-series-reformulated-multi}
than by positive integers as in \eqref{eq-intro-infinite-series}. 
The fact is that such a reparametrization has significant advantages. 
Among them, there is a much clearer dependence of $c_\alpha$ on $c_\beta$'s with $|\beta|<|\alpha|$ 
in the multi-index setting. 
In particular, when we expand $c_\alpha$ in terms of spherical harmonics, 
we are able to characterize the highest degree of spherical harmonics in $c_\alpha$ 
in terms of $\alpha$ explicitly. 

To end the introduction, we compare briefly the equations 
\eqref{eq-LiouvilleEquation} and \eqref{eq-Yamabe-equation}. 
Obviously, both equations are nonlinear in $u$ only, 
with \eqref{eq-LiouvilleEquation} in an exponential form 
and \eqref{eq-Yamabe-equation} in a power form. 
The exponential nonlinear term in \eqref{eq-LiouvilleEquation} 
generates an exponential decay factor in the cylindrical coordinates, 
as \eqref{eq-linearized-equation} illustrates. 
Such an exponential decay factor plays an essential role in our study. 
It easily eliminates the so-called {\it small divisors} 
in the derivation of the recursive estimates of $\{c_\alpha\}$. 
In contrast, the equivalent equation of \eqref{eq-Yamabe-equation}
in cylindrical coordinates does not have exponential decay factors, 
as shown in \cite{Caffarelli1989AsymptoticSA}, \cite{Korevaar1998RefinedAF}, 
and \cite{Han2019AsymptoticEO}. 
The elimination of the small divisors becomes highly nontrivial. 
We will study \eqref{eq-Yamabe-equation} in a forthcoming paper. 

The paper is organized as follows. 
In Section \ref{sec-Linear-Equations}, 
we rewrite \eqref{eq-LiouvilleEquation} in cylindrical coordinates 
and study the linearized equation. 
In Section \ref{sec-Formal-Expansions}, we study formal expansions of 
positive solutions of \eqref{eq-LiouvilleEquation} in terms of sequences of 
spherical harmonics. 
In Section \ref{sec-Asymp-Expansions}, we provide an alternative proof of Theorem B 
using our characterization in Section \ref{sec-Linear-Equations} and Section \ref{sec-Formal-Expansions}. 
In Section \ref{sec-Reparametrizations}, we reparametrize the formal expansions 
in terms of multi-indices with countably many components. 
In Section \ref{sec-Convergence-of-the-Expansions}, we prove the convergence of the formal expansions.

\section{Linear Equations}\label{sec-Linear-Equations}

In this section, we formulate the Liouville equation over cylindrical coordinates and study the corresponding linear equations. 

Following \cite{Caffarelli1989AsymptoticSA}, \cite{Korevaar1998RefinedAF}, and \cite{Han2019AsymptoticEO}, we introduce cylindrical coordinates $(t,\theta) \in [0,\infty) \times \mathbb{S}^1$ given by
$$t = -\ln |x|,\,\,\,\theta = \frac{x}{|x|}.$$
Let $u$ be a smooth solution to \eqref{eq-LiouvilleEquation} satisfying \eqref{eq-finite-exp-integral-condition} such that
$$u(x) - \nu_0 \ln |x| - a \rightarrow 0\quad\text{as }x \rightarrow 0,$$
for some $\nu_0 > -2$ and $a \in \R$. 
Set 
$$v(t,\theta) = u(x) - \nu_0\ln |x| - a.$$
It is easy to verify that $v$ satisfies
\begin{align}
\label{eq-linearized-equation}
    \partial_{tt} v + \partial_{\theta\theta} v  = e^{-\tau_0 t}f(v),
\end{align}
where $\tau_0 = 2 + \nu_0 > 0$ and
\begin{align*}
f(s) = - e^{a + s}.
\end{align*}
The specific expression of $f$ plays no roles. Throughout this paper, 
we assume that $f$ in \eqref{eq-linearized-equation} is a smooth or analytic function.

For simplicity, we write
\begin{align}
\label{eq-linear-operator}
L = \partial_{tt} + \partial_{\theta\theta}.
\end{align} 
Let $\{\lambda_i \}_{i\ge 0}$ be the sequence of distinct nonnegative eigenvalues 
of $-\partial_{\theta\theta}$, i.e., for $i\ge 0$,
\begin{align*}
\lambda_i= i^2.
\end{align*}
Note that we do not count multiplicity. 
In addition, we set, for $i\ge 0$, 
\begin{align*}
\text{$\mathscr{E}_i=$ the eigenspace of $-\partial_{\theta\theta}$ 
corresponding to $\lambda_i$.}
\end{align*}
Specifically, we set  
$$\mathscr{E}_0 = \{\text{constants}\},$$
and, for any $i \ge 1$,
$$\mathscr{E}_i = \R\text{-span}\{\cos(i\theta),\, \sin(i\theta)\}.$$
We point out that $\partial_{\theta\theta}$ is a differential operator on $\mathbb{S}^1$.

It is convenient to introduce the following terminology.

\begin{definition}\label{def-generalized-polynomial}
We say that $R \in C^\infty(\mathbb S^1)$ has a finite order if there exists a nonnegative integer $N$ such that
$$R(\theta) = \sum_{i=0}^N X_i (\theta),$$
for some $X_i \in \mathscr{E}_i$, $i = 0,\cdots, N$. The smallest possible $N$ is called the order of $R$.
\end{definition}

We first consider a special class of linear equations. Let $\mu>0$ be a constant. Consider
\begin{align*}
L (P(\theta) e^{-\mu t}) = R(\theta) e^{-\mu t},
\end{align*}
where $P$ and $R$ are smooth functions on $\mathbb S^1$. 
A simple computation yields
\begin{align}
\label{eq-linear-equation-mu}
\partial_{\theta\theta} P + \mu^2 P = R.
\end{align}
Note that \eqref{eq-linear-equation-mu} does not always have a solution. 
In the next result, we summarize two important cases that \eqref{eq-linear-equation-mu} is solvable. 
The orthogonal complement is taken with respect to the standard
inner product on $L^2(\mathbb{S}^1)$.

\begin{lemma}\label{lemma-polynomial_solution}
Let $\mu > 0$ be a constant and $R \in C^\infty (\mathbb S^1)$. 
\begin{itemize}
\item[(i)] If $\mu$ is not an integer, then 
\eqref{eq-linear-equation-mu} admits a unique solution $P \in C^\infty(\mathbb S^1)$.
\item[(ii)] 
If $\mu$ is an integer and $R \in \mathrm{Ker}\,(\partial_{\theta\theta} + \mu^2)^\perp$, 
then \eqref{eq-linear-equation-mu} admits a unique solution
$Q \in \mathrm{Ker}\,(\partial_{\theta\theta} + \mu^2)^\perp$, and
any solution $P$ of \eqref{eq-linear-equation-mu} can be written as 
$$P=X+Q,$$ 
for some $X \in \mathscr{E}_\mu$. 
\end{itemize}
Moreover, if $R$ has a finite order,  with its order strictly less than $\mu$ in $\mathrm{(ii)}$,
then the solutions $P$ in $\mathrm{(i)}$ and
$Q$ in $\mathrm{(ii)}$ have the same order as $R$.
\end{lemma}

\begin{proof} 
Note that $\mathrm{Ker}\,(\partial_{\theta\theta} + \mu^2) \ne 0$ 
if and only if $\mu$ is an integer. 
Hence, (i) and (ii) follow from the Fredholm alternative.

For the last assertion, assume that $R$ can be written as
$$R = \sum_{i=0}^N X_i,$$
for some $X_i \in \mathscr{E}_i$, $i = 0,\cdots, N$. 
Under the assumption in (i) or (ii), $\mu \ne i$ for $i = 0,\cdots, N$.
Hence, $P$ in (i) and $Q$ in (ii) are given by
$$\sum_{i=0}^N (\mu^2 - i^2)^{-1}X_i.$$
This yields the desired result. 
\end{proof}

Next, we discuss decay estimates for linear equations. 
Consider
\begin{align}
\label{eq-linear-equation-of-L}
Lv = \varphi\quad\text{in }[t_0,\infty)\times\mathbb{S}^1,
\end{align}
where $L$ is given by \eqref{eq-linear-operator}. 
In \cite{Han2019AsymptoticEO}, the following result is proved.

\begin{lemma}
\label{lemma-same-decay-estimate-for-L}
Let $\gamma > 0$, $k\geq 0$ be an integer, and $\varphi$ be a smooth function 
on $[t_0,\infty)\times\mathbb{S}^1$ for some $t_0>0$ satisfying, 
for any $(t_,\theta) \in [t_0,\infty)\times \mathbb{S}^1$,
$$|\varphi(t,\theta)| \leq C_0t^k e^{-\gamma t},$$
for some positive constant $C_0$. 
Let $v$ be a solution of \eqref{eq-linear-equation-of-L} on $[t_0,\infty)\times \mathbb{S}^1$ 
such that $v(t,\theta) \rightarrow 0$ as $t \rightarrow \infty$, uniformly in $\theta \in \mathbb{S}^1$.
\begin{itemize}
\item[(i)] If $\gamma \leq 1$, then there exists a constant $C>0$ such that, 
for any $(t,\theta) \in [t_0,\infty)\times \mathbb{S}^1$,
$$\left|v(t,\theta)\right| \leq 
\begin{cases} 
Ct^k e^{-\gamma t} & \text{if }  \gamma < 1,\\
C t^{k+1} e^{-\gamma t}& \text{if } \gamma = 1.
\end{cases}$$
\item[(ii)] If $\ell < \gamma \leq \ell+1$ for some integer $\ell \geq 1$, 
then there exist a constant $C>0$ and spherical harmonics $X_1 \in \mathscr{E}_1,\cdots,X_\ell \in \mathscr{E}_\ell$  
such that, for any $(t,\theta) \in [t_0,\infty)\times \mathbb{S}^1$,
$$\Big|v(t,\theta) - \sum_{i=1}^\ell X_i(\theta) e^{-i t} \Big| \leq 
\begin{cases} 
Ct^k e^{-\gamma t} & \text{if }  \ell < \gamma < \ell+1,\\
C t^{k+1} e^{-\gamma t}& \text{if } \gamma = \ell+1.
\end{cases}$$        
\end{itemize}
\end{lemma}

By Lemma \ref{lemma-same-decay-estimate-for-L}, 
if $v$ is a solution of \eqref{eq-linearized-equation} such that 
$v(t,\theta)\rightarrow 0$ as $t \rightarrow \infty$, uniformly in $\theta \in \mathbb{S}^1$,
then 
$$|v(t,\theta)| \le C e^{-\gamma t},$$
for some constants $C>0$ and $\gamma > 0$ with $\gamma < \min\{1,\tau_0\}$. 
Throughout this paper, we always consider exponential decay solutions of \eqref{eq-linearized-equation}.

\section{Formal Expansions}\label{sec-Formal-Expansions}

In this section, we construct formal expansions 
from a collection of spherical harmonics. 
There are two purposes for such formal expansions.
First, we construct an index set that provides the correct exponential decay orders.
Second, we derive recursive equations characterizing the terms in the expansion.
We need to point out that no results in later sections are proved based on 
the assumption of the formal expansions. 

We first introduce the index set that provides decay orders in the asymptotic expansions.
Let $\tau_0$ be the positive constant in \eqref{eq-linearized-equation}. Set
$$\rho_0 = \tau_0,$$
and, for each $i \ge 1$,
$$\rho_i = i.$$
Note that $\{\rho_i\}_{i\ge 0}$ may not be in increasing order. 
In fact, if $\tau_0$ is an integer, then $\rho_0 = \rho_i$ for some $i \ge 1$.
Throughout the paper, we denote by $\Z_+$ the collection of nonnegative integers. 

\begin{definition}\label{def-index-set}
The {\it index set} $\mathcal I$ associated with the operator $L$ in \eqref{eq-linear-operator} and $\tau_0 > 0$
is the collection 
of all nontrivial linear combinations of $\{\rho_i\}_{i \geq 0}$ with nonnegative integer coefficients, i.e., 
\begin{align*}
\mathcal I=\Big\{\sum_{i\geq 0}m_i\rho_i ;\, m_i 
\in \mathbb Z_+ \text{ with finitely many }m_i > 0\Big\}.
\end{align*}
Denote by $\mathcal I=\{\mu_\ell\}_{\ell\ge 1}$ a strictly increasing sequence of positive constants. 
\end{definition}

We now perform a formal computation to illustrate that $\mathcal I$ provides all exponential decay orders. 
Let $\{\mu_\ell\}_{\ell\geq 1}$ be the index set as in Definition \ref{def-index-set}. 
Assume $v$ is a solution of \eqref{eq-linearized-equation} on $[t_0,\infty)\times\mathbb{S}^1$ for some $t_0 \ge 0$ 
such that $v(t,\theta)$ converges to $0$ exponentially 
as $t\to\infty$, uniformly in $\theta \in \mathbb{S}^1$.
Formally, we write 
\begin{equation}\label{eq-formal-expansion}v(t,\theta)= 
\sum_{\ell=1}^\infty P_\ell(\theta) e^{-\mu_\ell t},\end{equation}
where $\{P_\ell\}_{\ell \geq 1}$ is a sequence of smooth functions on $\mathbb S^1$. 
Next, we expand $f$ in \eqref{eq-linearized-equation} formally by its Taylor series at $0$ 
$$f(s)=\sum_{j=0}^\infty a_j s^j.$$
By a formal term-by-term differentiation in \eqref{eq-formal-expansion}, a simple substitution, and rearrangements, we obtain
\begin{equation}
\label{eq-formal-computation}
\begin{aligned}
&\sum_{\ell=1}^\infty \big( \partial_{\theta\theta} P_\ell + \mu_\ell^2 P_\ell  \big) 
e^{-\mu_\ell t} \\
&\quad = a_0 e^{-\rho_0 t} 
+ \sum_{\ell = 1}^\infty \Big( \sum_{j=1}^\ell a_j \sum_{\mu_{k_1} + \cdots + \mu_{k_j} = \mu_\ell} P_{k_1}\cdots P_{k_j} \Big) 
e^{-(\mu_\ell + \rho_0) t}.
\end{aligned}
\end{equation}
We point out that the last summation is over all possible $(k_1,\cdots,k_j) \in \N^j$. 
Note that $\{\rho_0,\,\rho_0+\mu_1,\,\rho_0 + \mu_2,\,\cdots\}$ is strictly increasing.
We emphasize that each $P_\ell$ is assumed to be a smooth function on $\mathbb S^1$. 

First, recall the following decomposition lemma for spherical harmonics.

\begin{lemma}
\label{lemma-decomposition-of-spherical-harmonics}
Let $X\in \mathscr{E}_{m_1}$ and $Y \in \mathscr{E}_{m_2}$. 
Then, $XY \in \bigoplus_{i=0}^{m_1+m_2} \mathscr{E}_i$.
\end{lemma}

Here, the notation $Z \in \bigoplus_{i=0}^N \mathscr{E}_i$, $N \in \Z_+$, 
means that $Z$ can be written as a linear combination of functions in 
$\mathscr{E}_0, \cdots, \mathscr{E}_N$. 
We point out that Lemma \ref{lemma-decomposition-of-spherical-harmonics} is trivial for 
$\mathbb S^1$
and follows from well-known trigonometric formulas.

By Lemma \ref{lemma-polynomial_solution} and a simple comparison of decay orders, we obtain the following conclusion. 

\begin{proposition}\label{prop-polynomial-coefficients}
Let $\{\mu_\ell\}_{\ell\geq 1}$ be the index set associated with $L$ and $\tau_0>0$ as in Definition \ref{def-index-set} 
and $v$ be a solution 
of \eqref{eq-linearized-equation} on $[t_0,\infty)\times\mathbb{S}^1$ for some $t_0\ge 0$ 
such that $v$ converges to $0$ exponentially as 
$t \to \infty$, uniformly in $\theta \in \mathbb{S}^1$. 
Assume that $v$ is given by \eqref{eq-formal-expansion} formally 
for a sequence $\{P_\ell\}_{\ell\geq 1}$ of smooth functions on $\mathbb S^1$. Then, 
there exists a sequence of spherical harmonics $\{X_i\}_{i \ge 1}$ with $X_i\in \mathscr{E}_i$ such that 
\begin{itemize}
\item[(1)] if $\mu_\ell < \rho_0$, then $\mu_\ell = \ell$ and $P_\ell(\theta) = X_\ell(\theta)$; 
\item[(2)] if $\mu_\ell = \rho_0$, then
\begin{align*}
P_\ell (\theta) = 
\begin{cases}
    a_0/\rho_0^2 &\text{if }\rho_0\text{ is not an integer,}\\
    a_0/\rho_0^2 + X_{\rho_0}(\theta) &\text{if }\rho_0\text{ is an integer;}
\end{cases}
\end{align*}
\item[(3)] if $\mu_\ell > \rho_0$ and $\mu_\ell \neq \rho_0 + \mu_m$ for any $m \ge 1$, then $\mu_\ell$ is an integer and 
$$P_{\ell}(\theta) = X_{\mu_\ell}(\theta);$$
\item[(4)] if $\mu_\ell > \rho_0$ and $\mu_\ell = \rho_0 + \mu_m$ for some $1 \le m < \ell$, then
\begin{itemize} 
\item[(4a)] if $\mu_\ell$ is not an integer, then $P_\ell$ is the unique solution of
\begin{align}
\label{eq-recursive-equation-Asymp-expans}
 \partial_{\theta\theta} P_\ell + \mu_\ell^2 P_\ell 
 = \sum_{j=1}^m a_j \sum_{\mu_{k_1} + \cdots + \mu_{k_j} = \mu_m} P_{k_1}\cdots P_{k_j};
\end{align}
\item[(4b)] if $\mu_\ell$ is an integer, 
then 
$$P_\ell(\theta)=X_{\mu_\ell}(\theta)+Q_\ell(\theta),$$ 
where $Q_\ell$ is a solution 
of \eqref{eq-recursive-equation-Asymp-expans}, uniquely determined by $P_j$ for $1\le j<\ell$
under the additional requirement that $Q_\ell$ is perpendicular to $\mathscr{E}_{\mu_\ell}$ in $L^2(\mathbb S^1)$. 
\end{itemize}
\end{itemize}
Moreover, $P_\ell$ has an order at most $[\mu_\ell]$, the largest integer smaller than or equal to $\mu_\ell$.
\end{proposition} 

By the last assertion and Lemma \ref{lemma-decomposition-of-spherical-harmonics},
the right-hand side of \eqref{eq-recursive-equation-Asymp-expans} has an
order at most $[\mu_m] < \mu_\ell$. 
Therefore, Lemma \ref{lemma-polynomial_solution} is applicable to Case (4).

\begin{proof} We will use induction to characterize $P_\ell$ and estimate its order. 
Write
$$\mu_\ell = \alpha_0\rho_0 + \alpha_1\rho_1 + \cdots = \alpha_0\rho_0 + 1\alpha_1 + 2\alpha_2 + \cdots,$$
for some nonnegative integer sequence $\{\alpha_i\}_{i\ge 0}$ with finitely many nonzero terms. 

First, assume $\mu_\ell < \rho_0$. Then, $\alpha_0 = 0$. 
Hence, $\mu_\ell=\ell$, an integer, and the decay order $e^{-\mu_\ell t}$ 
does not appear in the right-hand side of \eqref{eq-formal-computation}. Therefore, for such $\ell$,
\begin{align}\label{eq-projection-ell} \partial_{\theta\theta} P_\ell + \mu_\ell^2 P_\ell = 0.\end{align}
In other words, $P_\ell \in \mathrm{Ker}\,(\partial_{\theta\theta} + \mu_\ell^2) =  \mathscr{E}_{\ell}$. This proves (1).

Second, assume $\mu_\ell > \rho_0$ and $\mu_\ell \neq \rho_0 + \mu_m$ for any $m \ge 1$. 
A simple argument yields $\alpha_0 = 0$. 
Hence, $\mu_\ell$ is an integer and the decay order $e^{-\mu_\ell t}$ 
does not appear in the right-hand side of \eqref{eq-formal-computation}. 
Therefore, \eqref{eq-projection-ell} holds for such $\ell$,
and hence $P_\ell \in \mathrm{Ker}\,(\partial_{\theta\theta} + \mu_\ell^2) =  \mathscr{E}_{\mu_\ell}$. This proves (3).

Next, assume that $\mu_\ell = \rho_0$. 
This corresponds to the least decay order in the right-hand side of \eqref{eq-formal-computation}. 
Hence, we have
$$\partial_{\theta\theta} P_\ell + \mu_\ell^2 P_\ell = a_0.$$
Since constant functions are orthogonal to $\mathrm{Ker}\,(\partial_{\theta\theta} + \rho_0^2)$ in $L^2(\mathbb S^1)$,
(2) follows from Lemma \ref{lemma-polynomial_solution}. 
Moreover, the order of $P_\ell$ is 0 if $\rho_0$ is not an integer 
and is at most $\rho_0$ if $\rho_0$ is an integer. 

Note that in either case of (1)-(3),  $P_\ell$ has an order at most $[\mu_\ell]$.
In addition, we obtain the initial step of the induction since $\mu_1 \le \rho_0$.

Last, consider $\mu_\ell > \rho_0$ and $\mu_\ell = \rho_0 + \mu_m$ for some $1 \le m < \ell$. 
In this case, both sides of \eqref{eq-formal-computation} contain 
the decay order $e^{-\mu_\ell t}$.
Hence, \eqref{eq-recursive-equation-Asymp-expans} holds for $P_\ell$.
By Lemma \ref{lemma-decomposition-of-spherical-harmonics} and the induction hypothesis, 
if $\mu_{k_1} + \cdots + \mu_{k_j} = \mu_m$ for some $1 \le j \le m$, then $a_j P_{k_1}\cdots P_{k_j}$
has an order at most $[\mu_{k_1}] + \cdots + [\mu_{k_j}] \le \mu_m$.
Therefore, the right-hand side of \eqref{eq-recursive-equation-Asymp-expans} has an order at most $[\mu_m]$.
Note that $[\mu_m] \le \mu_m < \mu_\ell$.
Therefore, by Lemma \ref{lemma-polynomial_solution}, we obtain that 
if $\mu_\ell$ is not an integer, then $P_\ell$ is the unique solution of \eqref{eq-recursive-equation-Asymp-expans}
and if $\mu_\ell$ is an integer, then 
$$P_\ell=X_{\mu_\ell}+Q_\ell,$$ 
where $X_{\mu_\ell} \in \mathscr E_{\mu_\ell}$ and $Q_\ell$ is a solution 
of \eqref{eq-recursive-equation-Asymp-expans}, uniquely determined by $P_j$ for $1\le j<\ell$
under the additional requirement that $Q_\ell$ is perpendicular to $\mathscr{E}_{\mu_\ell}$ in $L^2(\mathbb S^1)$.
In the first case that $\mu_\ell$ is not an integer, 
$P_\ell$ has an order at most $[\mu_m]$. 
In the second case that $\mu_\ell$ is an integer, $Q_\ell$ has an order at most $[\mu_m]$, and hence 
$P_\ell = X_{\mu_\ell}+Q_\ell$ has an order at most $\mu_\ell$.
In either case, $P_\ell$ has an order at most $[\mu_\ell]$.
This finishes the proof by induction.
\end{proof} 

We now make some remarks. First, Case (1) occurs if $\rho_0 > 1$, 
Case (3) never happens if $\rho_0$ is an integer, 
Case (4a) is relevant only when $\rho_0$ is not an integer, 
and Case (4b) never happens if $\rho_0$ is an irrational number.
Second, in the case that $\mu_\ell = \rho_0 + \mu_m$ for some $1 \le m < \ell$,
the right-hand side of \eqref{eq-recursive-equation-Asymp-expans} always has an order
strictly less than $\mu_\ell$.
In particular, we can solve \eqref{eq-recursive-equation-Asymp-expans} and 
characterize the solutions by Lemma \ref{lemma-polynomial_solution}. 
This is due to the extra exponential factor $e^{-\tau_0 t}$ in the right-hand side 
of \eqref{eq-linearized-equation} and is significantly different from the study of \eqref{eq-Yamabe-equation}. 

For brevity, given a sequence $\{X_i\}_{i \geq 1}$ 
with $X_i\in\mathscr{E}_i$, we write 
$$\mathcal{X} = \{X_i\}_{i \geq 1}.$$
In the process of determining the entire sequence $\{P_\ell\}_{\ell\ge 1}$, there are infinitely many free elements given by 
$ \{X_i\}_{i \geq 1}$ with $X_i\in\mathscr{E}_i$. 
They of course depend on $v$, but have no explicit expressions in terms of earlier $P_j$'s. 
If we know $\{X_i\}_{i \geq 1}$ by other methods, then we know the entire $\{P_\ell\}_{\ell \geq 1}$. 
In this sense, we write $\{P_\ell\}_{\ell \ge 1} = \{P_\ell(\mathcal{X})\}_{\ell\ge 1}$ to emphasize the dependence on $\mathcal{X}$.
The following result can be viewed as the converse of Proposition \ref{prop-polynomial-coefficients}.

\begin{proposition}\label{prop-generated-polynomials} 
Let $\{\mu_\ell\}_{\ell\geq 1}$ be the index set associated with $L$ and $\tau_0>0$ as in Definition \ref{def-index-set} 
and $\mathcal{X} =\{X_i\}_{i \geq 1}$ be a given sequence of spherical harmonics with $X_i\in \mathscr{E}_i$. 
Then, there exists a unique sequence of smooth functions $\{P_\ell(\mathcal{X})\}_{\ell\geq 1}$ on $\mathbb S^1$ such that 
Proposition \ref{prop-polynomial-coefficients}(1)-(4) hold and that 
$P_\ell$ has an order at most $[\mu_\ell]$.
\end{proposition} 

The proof can be adapted from that of Proposition \ref{prop-polynomial-coefficients} 
and is omitted.

With the sequence $\{P_\ell(\mathcal{X})\}_{\ell\ge 1}$, we can form a formal infinite series 
\begin{equation}\label{eq-formal-series-smooth}\sum_{\ell=1}^\infty P_\ell(\mathcal{X})(\theta) e^{-\mu_\ell t}.\end{equation}
We will refer to the series as the {\it formal expansion determined by $\mathcal{X}$}.  
We point out again that the purpose of formulating Proposition \ref{prop-polynomial-coefficients} 
is to summarize the properties of $\{P_\ell\}_{\ell\ge1}$. 
No results to be proved are based on the formal expansion \eqref{eq-formal-expansion} 
or \eqref{eq-formal-series-smooth}.

\section{Asymptotic Expansions}\label{sec-Asymp-Expansions}

In this section, we characterize asymptotic behaviors of exponential decay solutions 
of \eqref{eq-linearized-equation}
and derive asymptotic expansions. 
We will verify that the formal expansions derived in the previous section 
provide a good approximation of actual solutions. 
For a given exponential decay solution $v$ of \eqref{eq-linearized-equation}, 
such an asymptotic expansion is a formal infinite series as in \eqref{eq-formal-series-smooth}, 
for a given positive, strictly increasing, and divergent sequence of $\{\mu_\ell\}_{\ell\ge 1}$. 
We will prove that the difference of $v$ and the $\ell$-th partial sum 
has an exponential decay order $\mu_{\ell+1}$, for any $\ell\ge 1$. 
We will establish Theorem B in the introduction in Theorem \ref{thrm-asymptotic-expansion}.
Our proof here is different from that in \cite{GuoWanYang2021CalVar}.

Let $\tau_0>0$ be a constant and $f$ in \eqref{eq-linearized-equation} 
be a smooth function on $(-\epsilon_0,\epsilon_0)$, for some $\epsilon_0>0$. 
Then, any solution $v$ of \eqref{eq-linearized-equation} is smooth.  
Solutions are assumed to exist in $[t_0,\infty) \times \mathbb{S}^1$, for some $t_0 \ge0$. 
In particular, $|v| <\epsilon_0$ on $[t_0,\infty)\times\mathbb{S}^1$. 
Throughout this paper, we always consider a solution 
$v$ of \eqref{eq-linearized-equation} on $[t_0,\infty)\times\mathbb{S}^1$, for some $t_0 \ge 0$, 
such that {\it $v(t,\theta)$ converges to $0$ exponentially as $t\to\infty$, uniformly in $\theta \in \mathbb{S}^1$}, i.e., 
for any $t\geq t_0$ and $\theta \in \mathbb{S}^1$,
\begin{equation}\label{eq-exp-convergence}
\left|v(t,\theta) \right| \leq C_* e^{-\gamma t},\end{equation} 
for some positive constants $C_*$ and $\gamma$.

For simplicity, we will use the big $O$ notation: 
$$\text{$h = O(g) \Longleftrightarrow |h(t,\theta)| \leq Cg(t)$ for any $t \geq t_0$, 
$\theta \in \mathbb{S}^1$},$$ 
for some constant $C > 0$.  
We adopt this notation if we do not need an explicit dependence of $C$.

The next theorem is the main result in this section, concerning the asymptotic expansion of $v$.

\begin{theorem}\label{thrm-asymptotic-expansion}
Let $f$ be a smooth function on $(-\epsilon_0,\epsilon_0)$, 
$\tau_0>0$ be a constant, 
and $\{\mu_\ell\}_{\ell\geq 1}$ be the index set associated with $L$ and $\tau_0>0$ as in Definition \ref{def-index-set}.
Assume that $v$ is a solution of \eqref{eq-linearized-equation} on $[t_0,\infty)\times\mathbb{S}^1$ for some $t_0\ge 0$
such that $v(t,\theta)$ converges to $0$ exponentially as $t\to\infty$, uniformly in $\theta \in \mathbb{S}^1$.
Then, there exists a sequence $\mathcal{X} = \{X_i\}_{i\geq 1}$ 
of spherical harmonics with $X_i \in \mathscr{E}_i$
such that, for any $\ell \ge 1$ and any $(t,\theta) \in [t_0,\infty)\times\mathbb{S}^1$, 
\begin{align}\label{eq-higher-order-expansion-inequality}
\Big|{v(t,\theta) - \sum_{j=1}^\ell P_j(\mathcal{X})(\theta) e^{-\mu_j t}}\Big| \leq C_{\ell} e^{-\mu_{\ell+1} t},
\end{align}
where $\{P_\ell(\mathcal{X})\}_{\ell \geq 1}$ is the sequence of smooth functions on $\mathbb S^1$ 
determined by $\mathcal{X}$ as in Proposition \ref{prop-polynomial-coefficients}, 
and $C_{\ell}>0$ is a constant depending only on $v$ and $\ell$. 
\end{theorem}

\begin{proof}
All estimates in the proof hold for any $t \ge t_0$ and $\theta \in \mathbb S^1$.
For an arbitrarily fixed large integer $K$, we expand $f$ up to order $K$ and write 
\begin{align}\label{eq-RHS-expanded-K}
f(v)= \sum_{j=0}^K a_j v^j + f_K, 
\end{align}
where 
$$|f_K| \le C|v|^{K+1},$$ 
for some positive constant $C$ depending on the sup-norm of $f^{(K+1)}$ on $(-\epsilon_0,\epsilon_0)$. 
We first prove that there exists a sequence $\mathcal{X} = \{X_i\}_{i\geq 1}$ 
of spherical harmonics with $X_i \in \mathscr{E}_i$
such that, for any $\ell \ge 1$ and any $\mu$ with $\mu_\ell < \mu < \mu_{\ell+1}$,
\begin{align}\label{eq-higher-order-expansion-inequality-mu}
{v(t,\theta) - \sum_{j=1}^\ell P_j(\mathcal{X})(\theta) e^{-\mu_j t}} = O(e^{-\mu t}),
\end{align}
where $\{P_\ell(\mathcal{X})\}_{\ell \geq 1}$ is the sequence of smooth functions on $\mathbb S^1$
determined by $\mathcal{X}$ as in Proposition \ref{prop-polynomial-coefficients}. 

We will prove \eqref{eq-higher-order-expansion-inequality-mu} by induction. 
First, consider $\ell=1$. 
Note that $\mu_1 \le \rho_0$. By \eqref{eq-exp-convergence}, we have
$$Lv = e^{-\rho_0 t} f(v) = O(e^{-\rho_0 t}).$$ 
If $\mu_1 < \rho_0$, 
then $\mu_1 = 1$ and $e^{-\rho_0 t} f(v) = O(e^{-\mu t})$ for any $\mu \in (\mu_1,\mu_2)$.
By Lemma \ref{lemma-same-decay-estimate-for-L}(ii), there exists $X_1 \in \mathscr{E}_1$ such that, 
for any $(t,\theta) \in [t_0,\infty)\times\mathbb{S}^1$,
$$\big|v(t,\theta) - X_1(\theta) e^{-t}\big| \leq C_\mu e^{-\mu t}.$$
We point out that if $\mu \in (\mu_1,\mu_2) = (1,\mu_2)$, then $\mu$ is not an integer and $\mu < 2$.
In addition, $X_1$ is independent of $\mu$ between $\mu_1$ and $\mu_2$.
Next, assume that $\mu_1 = \rho_0$. In this case, $0 < \rho_0 \le 1$. Set
$$v_1 = v - (a_0/\rho_0^2) e^{-\rho_0 t}.$$
Then, by \eqref{eq-exp-convergence}, $v_1$ satisfies
$$L v_1 = e^{-\rho_0 t}\sum_{j=1}^K a_j v^j + e^{-\rho_0 t}f_K = O(e^{-(\rho_0 + \gamma) t}).$$
Choose $0 < \epsilon < \gamma$ small such that $\rho_0 + \epsilon < \mu_2$. 
Then, $\rho_0 + \epsilon<2$ is not an integer. If $\rho_0 + \epsilon <1$, 
by Lemma \ref{lemma-same-decay-estimate-for-L}(i), we get
$$v_1 = O(e^{-(\rho_0 + \epsilon) t}).$$
If $\rho_0 + \epsilon > 1$, Lemma \ref{lemma-same-decay-estimate-for-L}(ii) implies that
there exists $X_1 \in \mathscr{E}_1$ such that
$$v_1 - X_1(\theta) e^{-t} = O(e^{-(\rho_0 + \epsilon) t}).$$
Hence,
\begin{align*}
P_1(\theta) = 
\begin{cases}
a_0/\rho_0^2 &\text{if }\mu_1 = \rho_0 < 1,\\
a_0/\rho_0^2 + X_1(\theta)&\text{if }\mu_1 = \rho_0 = 1,
\end{cases}
\end{align*}
and
$$v = P_1(\theta) e^{-\mu_1 t} + O(e^{-(\mu_1 + \epsilon) t}).$$
In particular, $v = O(e^{-\mu_1 t})$ and 
$$L\big(v- P_1(\theta) e^{-\mu_1 t}\big) 
= e^{-\rho_0 t}\sum_{j=1}^K a_j v^j + e^{-\rho_0 t}f_K = O(e^{-(\rho_0 + \mu_1) t}) = O(e^{-\mu_2 t}).$$
Therefore, we can improve the decay to $v - P_1(\theta) e^{-\mu_1 t}=O(e^{-\mu t})$ 
for arbitrary $\mu$ between $\mu_1$ and $\mu_2$. 
This finishes the proof of \eqref{eq-higher-order-expansion-inequality-mu} for $\ell = 1$.

Now, assume that \eqref{eq-higher-order-expansion-inequality-mu} holds for some $\ell\ge 1$ and 
some smooth functions $P_1, \cdots, P_\ell$ on $\mathbb S^1$ as in Proposition \ref{prop-polynomial-coefficients}.
For simplicity, set
$$v_\ell = v- \sum_{i=1}^\ell P_i(\theta) e^{-\mu_i t}.$$ 
By the induction hypothesis, we get, for any $\mu_\ell < \mu < \mu_{\ell+1}$,
$$v_\ell = O(e^{-\mu t}).$$ 
Write 
$$v =v_\ell+ \sum_{i=1}^\ell P_i(\theta) e^{-\mu_i t},$$ 
and substitute this in \eqref{eq-linearized-equation}. 
With \eqref{eq-RHS-expanded-K}, we have 
\begin{equation}\label{eq-bootstrapping-eq}
L v_\ell + \sum_{i=1}^\ell L (P_i(\theta) e^{-\mu_i t})= F,
\end{equation}
where 
\begin{align}\label{eq-RHS-expanded-K-2}
F =  e^{-\rho_0 t}\sum_{j=0}^K a_j v^j + e^{-\rho_0 t}f_K.
\end{align}
Here, we choose sufficiently large $K>\ell$. We will construct a smooth function $Q_{\ell+1}$ on $\mathbb S^1$ 
such that $Q_{\ell+1}\perp \mathrm{Ker}\,(\partial_{\theta\theta}+ \mu_{\ell+1}^2)$ in $L^2(\mathbb{S}^1)$ and
\begin{align}
\label{eq-higher-order-estimate-inductive-step}
L(v_\ell - Q_{\ell+1}(\theta) e^{-\mu_{\ell+1} t}) = O(e^{-(\mu_{\ell+1} +\epsilon )t}),
\end{align}
for some $\epsilon > 0$ small.

We consider several cases. First, assume that $\mu_{\ell+1}\le \rho_0$. 
In this case, $\mu_i = i$ for $i = 1,\cdots,\ell$ and $P_i = X_i \in \mathscr{E}_i$ 
by Proposition \ref{prop-polynomial-coefficients}(1). Therefore, \eqref{eq-bootstrapping-eq} reduces to
$$L v_\ell = F,$$
with $F$ given by \eqref{eq-RHS-expanded-K-2}. If $\mu_{\ell+1} < \rho_0$, then $\rho_0 \ge \mu_{\ell+2}$ and
$$F = O(e^{-\rho_0 t}) =  O(e^{-\mu_{\ell+2} t}).$$
If $\mu_{\ell+1} = \rho_0$, then
$$L\big(v_\ell - a_0/\rho_0^2 e^{-\mu_{\ell+1} t}\big) 
= e^{-\rho_0 t}\sum_{j=1}^K a_j v^j + e^{-\rho_0 t}f_K = O(e^{-(\rho_0 + \mu_1) t}) = O(e^{-\mu_{\ell+2} t}),$$
where we used
$$L(a_0/\rho_0^2 e^{-\mu_{\ell+1} t}) = a_0 e^{-\rho_0 t},$$
and
$$\rho_0 + \mu_1 = \mu_{\ell+1} + \mu_1 \ge \mu_{\ell+2}.$$
In summary, if $\mu_{\ell+1} \leq \rho_0$, then \eqref{eq-higher-order-estimate-inductive-step} holds 
with $Q_{\ell+1}$ given by
\begin{align*}
Q_{\ell+1}=
\begin{cases}
0 &\text{if }\mu_{\ell+1} < \rho_0,\\
a_0/\rho_0^2 & \text{if }\mu_{\ell+1} = \rho_0.
\end{cases}
\end{align*}

Next, we consider $\mu_{\ell+1} > \rho_0$.
Note that, for any $j \ge 1$,
$$a_j\Big(v_\ell + \sum_{i=1}^\ell P_i(\theta) e^{-\mu_i t}\Big)^j 
= \sum_{s = 0}^j \left(\begin{matrix}j\\s\end{matrix}\right) a_j v_\ell^s\cdot 
\Big(\sum_{i=1}^\ell P_i(\theta)e^{-\mu_i t}\Big)^{j-s} .$$
If $s \geq 1$, then 
$$a_j v_\ell^s\cdot \Big(\sum_{i=1}^\ell P_i(\theta)e^{-\mu_i t}\Big)^{j-s}=O(e^{-(\mu_\ell + \epsilon) t}),$$ 
for any $0 < \epsilon < \mu_{\ell+1} - \mu_{\ell}$. 
For  $s = 0$, we have
\begin{align*}
&\sum_{j=1}^K a_j\Big(\sum_{i=1}^\ell P_i(\theta) e^{-\mu_i t}\Big)^j\\
&\qquad = \sum_{i = 1}^{\ell}\Big(\sum_{j = 1}^i \sum_{\mu_{k_1} + \cdots + \mu_{k_j} 
= \mu_i} a_j P_{k_1} \cdots P_{k_j} \Big)e^{-\mu_i t} 
+ O(e^{-(\mu_{\ell}+\epsilon) t}).
\end{align*}
Substituting these estimates in \eqref{eq-RHS-expanded-K-2}, we get 
$$F = a_0 e^{-\rho_0 t} +\sum_{i = 1}^{\ell}\Big(\sum_{j = 1}^i \sum_{\mu_{k_1} + \cdots + \mu_{k_j} = \mu_i} 
a_j P_{k_1} \cdots P_{k_j} \Big)e^{-(\mu_i+\rho_0) t} 
+ O(e^{-(\mu_{\ell+1}+\epsilon ) t}),$$
where we used
$$\mu_\ell + \rho_0 \ge \mu_{\ell+1}.$$
For simplicity, set $\mu_0 = 0$.  Choose $1 \le m \le\ell$ such that $\rho_0 + \mu_{m-1} < \mu_{\ell+1} \leq \rho_0 + \mu_{m}$.
By Proposition \ref{prop-polynomial-coefficients} and the induction hypothesis,
we can reduce \eqref{eq-bootstrapping-eq} to 
\begin{align}\label{eq-bootstrapping-eq-2}
L v_\ell = \sum_{j = 1}^{m} \sum_{\mu_{k_1} + \cdots + \mu_{k_j} 
= \mu_{m}} a_j P_{k_1} \cdots P_{k_j} e^{-(\rho_0 + \mu_{m} )t} + O(e^{-(\mu_{\ell+1}+\epsilon)  t}).
\end{align}
If $\mu_{\ell+1} < \rho_0 + \mu_{m}$, then $\rho_0 +\mu_{m} \ge \mu_{\ell+2}$ 
and \eqref{eq-bootstrapping-eq-2} implies
$$Lv_\ell = O(e^{-(\mu_{\ell+1} + \epsilon) t}).$$
Set $Q_{\ell+1} = 0$ in this case. On the other hand, if $\mu_{\ell+1} = \rho_0 + \mu_{m}$, 
then \eqref{eq-bootstrapping-eq-2} yields
\begin{align*}
L v_\ell = \sum_{j = 1}^{m} \sum_{\mu_{k_1} + \cdots + \mu_{k_j} 
= \mu_{m}} a_j P_{k_1} \cdots P_{k_j} e^{-\mu_{\ell+1}t} + O(e^{-(\mu_{\ell+1} + \epsilon)  t}).
\end{align*}
Consider 
\begin{align*}
\partial_{\theta\theta} Q_{\ell+1} + \mu_{\ell+1}^2 Q_{\ell+1}
= \sum_{j = 1}^{m} \sum_{\mu_{k_1} + \cdots + \mu_{k_j} = \mu_{m}} a_j P_{k_1} \cdots P_{k_j}.
\end{align*} 
By Lemma \ref{lemma-decomposition-of-spherical-harmonics} and Proposition \ref{prop-polynomial-coefficients}, 
the right-hand side has an order at most $[\mu_m] < \mu_{\ell+1}$.
By Lemma \ref{lemma-polynomial_solution}, we conclude that 
there exists a unique smooth solution $Q_{\ell+1}$ such that $Q_{\ell+1}$ is perpendicular to 
$\mathrm{Ker}\, (\partial_{\theta\theta}  +  \mu_{\ell+1}^2)$ in $L^2(\mathbb{S}^1)$.
This finishes the proof of \eqref{eq-higher-order-estimate-inductive-step}.

Now, fix $\epsilon > 0$ small such that $ (\mu_{\ell+1}, \mu_{\ell+1} + \epsilon]$ does not contain any integer.
Let $n$ be the largest integer such that $n < \mu_{\ell+1} + \epsilon$. Then, $n \le \mu_{\ell+1}$.
By Lemma \ref{lemma-same-decay-estimate-for-L}(ii), 
there exist spherical harmonics $X_i\in\mathscr{E}_i$ for $i=1, \cdots, n$, 
such that 
$$v_\ell - Q_{\ell+1} (\theta) e^{-\mu_{\ell+1} t} - \sum_{i = 1}^n X_i(\theta) e^{-i t} 
= O(e^{-(\mu_{\ell+1} + \epsilon) t}).$$ 
First, assume that $\mu_{\ell+1}$ is not an integer. Then, $n \leq \mu_\ell$. 
Since $Q_{\ell+1}(\theta)e^{-\mu_{\ell+1} t}$ and $v_\ell$  
have decay order $ e^{-(\mu_\ell + \epsilon) t}$, 
we must have $X_1, \cdots, X_n = 0$. In this case, we set
$$P_{\ell+1} = Q_{\ell+1}.$$
Next, assume $n=\mu_{\ell+1}$. Then, the same argument shows that $X_1, \cdots, X_{n-1} = 0$. 
Set 
$$P_{\ell+1}(\theta) = X_n(\theta)+Q_{\ell+1}(\theta).$$ 
Therefore, in either case,
$$v_\ell - P_{\ell+1}(\theta) e^{-\mu_{\ell+1} t} = O(e^{-(\mu_{\ell+1} + \epsilon) t}),$$ 
for some $\epsilon > 0$ small. By a similar computation leading to \eqref{eq-bootstrapping-eq} 
and \eqref{eq-bootstrapping-eq-2} but with 
$v_{\ell+1}$ replacing $v_\ell$, we have  
$$Lv_{\ell+1} = O(e^{-\mu_{\ell+2} t}).$$ 
We can improve the decay to $v_{\ell+1} = O(e^{-\mu t})$ for any $\mu_{\ell+1}< \mu < \mu_{\ell+2}$ as before. 
This finishes the proof of \eqref{eq-higher-order-expansion-inequality-mu} by induction. 
Note that, from the construction, there exists a sequence $\mathcal{X} = \{X_i\}_{i\geq 1}$ 
of spherical harmonics such that 
$\{P_\ell\}_{\ell\ge 1} = \{P_\ell(\mathcal{X})\}_{\ell\geq 1}$ is the sequence of smooth functions on $\mathbb S^1$
determined by $\mathcal{X}$ as in Proposition \ref{prop-polynomial-coefficients}. 

Finally, we prove \eqref{eq-higher-order-expansion-inequality}. 
Fix $\ell \ge 1$. 
Choose $\epsilon>0$ small such that $\mu_{\ell+1} +\epsilon <\mu_{\ell+2}$. 
Then, \eqref{eq-higher-order-expansion-inequality-mu} implies
$$v - \sum_{j=1}^\ell P_j(\mathcal{X})(\theta) e^{-\mu_j t} - P_{\ell+1}(\mathcal{X})(\theta)e^{-\mu_{\ell+1} t} 
= O(e^{-(\mu_{\ell+1} + \epsilon) t}).$$
Hence, 
$$v - \sum_{j=1}^\ell P_j(\mathcal{X})(\theta) e^{-\mu_j t} 
= P_{\ell+1}(\mathcal{X})(\theta)e^{-\mu_{\ell+1} t} + O(e^{-(\mu_{\ell+1} + \epsilon) t}) = O(e^{-\mu_{\ell+1}  t}).$$
This proves \eqref{eq-higher-order-expansion-inequality}.
\end{proof}

A function $v$ defined on $[t_0, \infty)\times \mathbb{S}^1$ is said to 
have its {\it asymptotic behavior as $t\to\infty$ determined by $\mathcal{X} = \{X_i\}_{i \geq 1}$} 
if $v$ satisfies \eqref{eq-higher-order-expansion-inequality}, i.e, 
for any $\ell \ge 1$ and any $(t,\theta) \in [t_0,\infty)\times\mathbb{S}^1$,
\begin{align*}
\Big|{v(t,\theta) - \sum_{j=1}^\ell P_j(\mathcal{X})(\theta) e^{-\mu_j t}}\Big| \leq C_{\ell} e^{-\mu_{\ell+1} t},
\end{align*}
for some positive constant $C_{\ell}$. 
By Theorem \ref{thrm-asymptotic-expansion}, 
any solution $v$ of \eqref{eq-linearized-equation} 
with $v(t,\theta)$ converging to $0$ exponentially as $t\to\infty$, uniformly in $\theta \in \mathbb{S}^1$,
has its asymptotic behavior as $t\to\infty$ determined by some sequence $\{X_i\}_{i \geq 1}$ 
of spherical harmonics with $X_i\in\mathscr{E}_i$.

\section{Reparametrizations}
\label{sec-Reparametrizations}

In the previous sections, for an exponential decay solution $v$ to \eqref{eq-linearized-equation}, 
we derived its (formal) asymptotic expansion
\begin{align}
\label{eq-formal-expansion-reparam}
\sum_{\ell=1}^\infty P_\ell(\theta) e^{-\mu_\ell t}.
\end{align}
Recall that $\{P_\ell\}_{\ell\geq 1} = \{P_\ell(\mathcal{X})\}_{\ell\geq 1}$ 
is the sequence of smooth functions on $\mathbb S^1$ determined 
by a sequence $\mathcal{X} = \{X_i\}_{i\geq 1}$ of spherical harmonics
with $X_i\in\mathscr{E}_i$.
In this section, we reparametrize the expansion by appropriate multi-indices $\alpha$ 
and derive several properties that play important roles in our study.

First, we introduce the following index set for the multi-indices.

\begin{definition}
\label{def-index-set-multi-indices}
For each $i \ge 0$, denote by $e_i \in \Z_+^\infty$ the integer sequence whose $i$-th component is $1$ 
and all the other components are $0$. Define
$$\varLambda = \Z_+\text{-}\text{span}\{ e_i;\,i\geq 0\} \setminus\{0\}.$$
\end{definition}

In other words, $\varLambda$ consists of all non-zero sequences $\alpha = \{\alpha_i\}_{i\geq 0}$ 
such that $\alpha_i \in \Z_+$ for any $i\geq 0$ and $\alpha_i = 0$ for all but finitely many $i$. 
We emphasize that there are only finitely many nonzero components in each $\alpha \in \varLambda$. 
This is the main reason we introduce a new notation $\varLambda$ instead of using $\Z_+^{\infty}\setminus\{0\}$. 
Note that the index of $\alpha \in \varLambda$ starts from $0$.

For any $\alpha = \{\alpha_i\}_{i\geq 0} \in \varLambda$, we write
\begin{align}
\label{eq-norm-of-multi-indices}
|\alpha| = \sum_{i=0}^\infty \alpha_i,
\end{align}
and, for any real-valued sequence $x = \{x_i\}_{i\geq 0} \in \R^\infty$, we write
\begin{align}
\label{eq-inner-product-indices}
\langle \alpha, x \rangle = \sum_{i=0}^\infty \alpha_i x_i.
\end{align}
Since $\alpha_i = 0$ for all but finitely many $i$, there are finitely many terms in the sums in \eqref{eq-norm-of-multi-indices} 
and \eqref{eq-inner-product-indices}. In addition, we write
$$\alpha = (\alpha_0,\tilde\alpha),$$
where
$$\tilde \alpha = \{\alpha_i\}_{i \ge 1}.$$
In particular, with $\Z_+$ as the sequence of nonnegative integers and $\N$ as the sequence of positive integers, we have
$$\langle \alpha,\Z_+\rangle = \langle\tilde\alpha,\N\rangle = \sum_{i=1}^\infty i\alpha_i.$$

Throughout this section, we fix a constant $\tau_0 > 0$ 
and a smooth function $f$ on $(-\epsilon_0,\epsilon_0)$, for some $\epsilon_0> 0$, 
with the Taylor series at $0$ given by 
$$f(s) = \sum_{j=0}^\infty a_j s^j.$$
In this section, we do not impose the convergence of this series. 
We only need the collection of the coefficients $\{a_j\}_{j\ge0}$.

We now reparametrize $\{P_\ell(\mathcal{X})\}_{\ell\geq 1}$ by multi-indices $\alpha \in \varLambda$. 
For simplicity, we write 
$$\rho = \{\rho_i\}_{i \geq 0} = (\tau_0,\N).$$

\begin{proposition}
\label{prop-formal-expansion-multi-index}
Let $\{\mu_\ell\}_{\ell\geq 1}$ be the index set associated with $L$ and $\tau_0>0$ as in Definition \ref{def-index-set}, 
$\mathcal{X} =\{X_i\}_{i \geq 1}$ be a sequence of spherical harmonics
with $X_i \in \mathscr{E}_i$, 
and $\{P_\ell\}_{\ell\geq 1}$ be the sequence of smooth functions on $\mathbb S^1$ determined 
by $\mathcal{X}$ as in Proposition \ref{prop-generated-polynomials}.
Define $P_{e_0} = a_0/\rho_0^2$ and $P_{e_i} = X_i$ for $i\geq 1$, and $P_\alpha$ inductively 
for $\alpha= \{\alpha_i\}_{i\geq 0} \in \varLambda$ with $|\alpha| \geq 2$ as follows:
\begin{itemize}
\item[(i)] if $\alpha_0 = 0$, then $P_\alpha = 0$;
\item[(ii)] if $\alpha_0 \ge 1$, then $P_\alpha$ is the unique 
solution of
\begin{align}
\label{eq-recursive-equation-multi-index}
  \partial_{\theta\theta} P_\alpha + \langle \alpha,\rho\rangle^2 P_\alpha 
= \sum_{j = 1}^{|\alpha|-1} \sum_{\beta_1 + \cdots + \beta_j = \alpha - e_0} a_j P_{\beta_1} \cdots P_{\beta_j},
\end{align}
with $P_\alpha$ perpendicular to 
$\mathrm{Ker}\, (\partial_{\theta\theta} + \langle \alpha,\rho\rangle^2 )$ in $L^2(\mathbb{S}^1)$. 
\end{itemize}
Then, $P_\alpha \in C^\infty (\mathbb S^1)$ has an order at most $\langle\alpha,\Z_+\rangle$.
Moreover, for each $\ell \ge 1$,
$$P_\ell = \sum_{\alpha, \langle \alpha, \rho \rangle = \mu_\ell } P_\alpha.$$
\end{proposition}

We point out that, by Lemma \ref{lemma-decomposition-of-spherical-harmonics}, 
the right-hand side of \eqref{eq-recursive-equation-multi-index} has an
order at most $\langle \alpha,\Z_+\rangle < \langle \alpha,\rho\rangle$, 
since the $0$-th coordinate of $\Z_+$ is $0$.
Therefore, Lemma \ref{lemma-polynomial_solution} can be applied to solve \eqref{eq-recursive-equation-multi-index}.
We can say that $\{P_\alpha\}_{\alpha \in \varLambda}$ 
is the sequence of smooth functions on $\mathbb S^1$
determined by $\mathcal{X} =\{X_i\}_{i \geq 1}$
and write $P_\alpha(\mathcal X) = P_\alpha$ for the dependence on $\mathcal X$.

\begin{proof}
By a similar inductive argument as in the proof of 
Proposition \ref{prop-polynomial-coefficients}, we obtain the assertion on the existence of $P_\alpha$ 
and its order.
We omit the details.

Set 
\begin{align}\label{eq-definition-R-ell}R_\ell = \sum_{\alpha, \langle \alpha, \rho \rangle = \mu_\ell } P_\alpha.\end{align} 
We will use induction on $\ell$ to prove, for any $\ell \ge 1$,
\begin{align}
\label{eq-R_ell = P_ell}
P_\ell = R_\ell.
\end{align}
In the summation in \eqref{eq-definition-R-ell}, we only need to consider 
$\alpha$ with $|\alpha|=1$ or $|\alpha|\ge 2$ and $\alpha_0\ge 1$ since 
$P_\alpha=0$ for $\alpha$ with $|\alpha|\ge 2$ and $\alpha_0=0$.

First, assume $\mu_\ell < \rho_0$. 
In this case, $\mu_\ell = \ell$ and $P_\ell=X_\ell$ by Proposition \ref{prop-polynomial-coefficients}(1). 
Consider $\alpha = \{\alpha_i\}_{i\geq 0} \in \varLambda$ with $\langle\alpha,\rho\rangle=\mu_\ell = \ell$.
Then, $\alpha_0 = 0$. 
If $|\alpha|=1$, we have only one choice of such $\alpha$ given by $\alpha=e_{\ell}$. 
Hence, $\alpha = e_\ell$ is the only choice of $\alpha$ in the summation in \eqref{eq-definition-R-ell}, and then 
$R_\ell=P_{e_{\ell}}=X_\ell.$
We thus have \eqref{eq-R_ell = P_ell}.

Next, assume $\mu_\ell = \rho_0$. Then, Proposition \ref{prop-polynomial-coefficients}(2) yields
\begin{align*}
P_\ell (\theta) = a_0/\rho_0^2 \quad\text{if }\rho_0\text{ is not an integer,}
\end{align*}
and
\begin{align*}
P_\ell (\theta) = a_0/\rho_0^2 + X_{\rho_0}(\theta) \quad\text{if }\rho_0\text{ is an integer}.
\end{align*}
If $\rho_0$ is not an integer, we have
$\langle\alpha,\rho\rangle = \rho_0$ if and only if $\alpha = e_0.$
In other words, $\alpha = e_0$ is the only choice of $\alpha$ in \eqref{eq-definition-R-ell}.
Hence, $R_\ell = P_{e_0} = a_0/\rho_0^2$. 
If $\rho_0 = n$ is an integer, then $\langle \alpha,\rho\rangle = \rho_0$ 
if and only if $\alpha = e_0,e_n$, or $|\alpha|\ge 2$ and $\alpha_0 = 0$. 
Hence,
$$R_\ell = P_{e_0} + P_{e_n} = a_0/\rho_0^2 + X_n.$$
In either case, \eqref{eq-R_ell = P_ell} holds. 

Note that the discussion above yields \eqref{eq-R_ell = P_ell} for $\ell = 1$, since $\mu_1 \le \rho_0$.
Now, we assume that $\mu_\ell > \rho_0$ and \eqref{eq-R_ell = P_ell} holds for $1,\cdots,\ell-1$. 
By Proposition \ref{prop-polynomial-coefficients}, there are three cases.

First, assume $\mu_\ell \neq \rho_0 + \mu_m$ for any $m \ge 1$. 
By Proposition \ref{prop-polynomial-coefficients}(3), $\mu_\ell = n$ is an integer and $P_\ell = X_n$. 
As in the case $\mu_\ell < \rho_0$ above, if $\langle \alpha,\rho\rangle = \mu_\ell = n$, 
either $\alpha = e_n$ or $|\alpha|\ge 2$ and $\alpha_0 = 0$. 
Hence, $\alpha = e_n$ is the only choice of $\alpha$ in the summation in \eqref{eq-definition-R-ell}, and then 
$R_\ell=P_{e_{n}}=X_n.$ We thus have \eqref{eq-R_ell = P_ell}.

Next, assume $\mu_\ell = \rho_0 + \mu_m$ for some $1 \le m <\ell$ and $\mu_\ell$ is not an integer. 
By \eqref{eq-definition-R-ell}, we get 
$$\partial_{\theta\theta} R_\ell + \mu_\ell^2 R_\ell = \sum_{\alpha, \langle \alpha, \rho \rangle 
= \mu_\ell }  \big(\partial_{\theta\theta} P_\alpha + \langle \alpha, \rho \rangle^2 P_\alpha \big).$$
Since $\mu_\ell$ is not an integer, $\langle \alpha, \rho \rangle = \mu_\ell$ implies $\alpha_0\ge 1$. 
In particular, $\alpha - e_0$ consists of nonnegative terms.
By \eqref{eq-recursive-equation-multi-index} for each such $\alpha$, we have
\begin{align*}
\partial_{\theta\theta} R_\ell + \mu_\ell^2 R_\ell 
= \sum_{\alpha, \langle \alpha, \rho \rangle = \mu_\ell } \sum_{j = 1}^{|\alpha|-1} 
\sum_{\beta_1 + \cdots + \beta_j = \alpha - e_0} a_j P_{\beta_1} \cdots P_{\beta_j}.
\end{align*}
Here, we sum over all $\beta_1,\cdots,\beta_j \in \varLambda$ 
such that $\beta_1 + \cdots + \beta_j = \alpha - e_0$ for some $\alpha \in \varLambda$ 
satisfying $\langle \alpha, \rho \rangle = \mu_\ell$. 
Note that $\langle \alpha, \rho \rangle = \mu_\ell$ if and only if
$$\langle\alpha-e_0,\rho\rangle = \mu_\ell - \rho_0 = \mu_m.$$
Hence, we sum over all $\beta_1,\cdots,\beta_j \in \varLambda$ 
such that $\beta_1 + \cdots + \beta_j$ 
lies in the hyperplane $\{x \in \R^{\infty};\langle x,\rho \rangle = \mu_m\}.$
In other words, we sum over all $\beta_1,\cdots,\beta_j \in \varLambda$ such that 
$\langle \beta_1,\rho \rangle + \cdots + \langle \beta_j,\rho \rangle 
= \mu_m$, for any possible $j \ge 1$. Note that each $\langle\beta_s,\rho\rangle$ is some $\mu_{i_s}$. 
Since 
$$\langle \beta_1,\rho \rangle + \cdots + \langle \beta_j,\rho \rangle \geq j \mu_1 \ge \mu_j,$$ 
we have $j \leq m$. 
Therefore,
\begin{align*}
\sum_{\alpha, \langle \alpha, \rho \rangle = \mu_\ell } \sum_{j = 1}^{|\alpha|-1} 
\sum_{\beta_1 + \cdots + \beta_j = \alpha - e_0} = \sum_{j=1}^m 
\sum_{\substack{\beta_1,\cdots,\beta_j\\ \langle \beta_1,\rho \rangle + \cdots + \langle \beta_j,\rho \rangle = \mu_m}}.
\end{align*}
By a simple rearrangement, we obtain
\begin{align*}
\partial_{\theta\theta} R_\ell + \mu_\ell^2 R_\ell 
&= \sum_{j=1}^m \sum_{\mu_{k_1} + \cdots + \mu_{k_j} = \mu_m} 
a_j\Big(\sum_{\beta_1, \langle \beta_1,\rho \rangle = \mu_{k_1} } 
P_{\beta_1} \Big)\cdots \Big( \sum_{\beta_j, \langle \beta_j,\rho \rangle = \mu_{k_j} } P_{\beta_j}\Big)\\
&= \sum_{j=1}^m \sum_{\mu_{k_1} + \cdots + \mu_{k_j} = \mu_m} a_j R_{k_1} \cdots R_{k_j}.
\end{align*}
In the summation in the right-hand side above, each $k_i$ satisfies $k_i <\ell$.
By the induction hypothesis, we have $P_{k_i}=R_{k_i}$ for each $i=1, \cdots, j$ as above, and hence 
$$\partial_{\theta\theta}R_\ell + \mu_\ell^2 R_\ell
=\sum_{j=1}^m \sum_{\mu_{k_1} + \cdots + \mu_{k_j} = \mu_m} a_j P_{k_1} \cdots P_{k_j}.$$
Since $\mu_\ell$ is not an integer, Lemma \ref{lemma-polynomial_solution}(i) 
and Proposition \ref{prop-polynomial-coefficients}(4a) yield \eqref{eq-R_ell = P_ell}.

Finally, assume $\mu_\ell = \rho_0 + \mu_m$ for some $1 \le m <\ell$ and $\mu_\ell = n$ is an integer. 
By \eqref{eq-definition-R-ell}, we have 
$$R_\ell = P_{e_n} +  \sum_{\substack{|\alpha|\geq 2, \,\alpha_0\ge 1,\\ \langle \alpha, \rho \rangle = \mu_\ell} } 
P_\alpha  = X_n +  S_\ell.$$  
By (ii), each $P_\alpha$ in the summation above is perpendicular to 
$ \mathrm{Ker}\, (\partial_{\theta\theta} + \mu_\ell^2) = \mathscr{E}_n$ in $L^2(\mathbb{S}^1)$,
and then $S_\ell$ is perpendicular to 
$\mathscr{E}_n$ in $L^2(\mathbb{S}^1)$. 
Note that $\langle \alpha, \rho \rangle = \mu_\ell$ with $|\alpha| \ge 2$ and $\alpha_0 \ge 1$ if and only if
$\alpha - e_0 = \alpha'$ for some $\alpha' \in \varLambda$ with $|\alpha'|\ge 1$ and $\langle\alpha',\rho\rangle = \mu_m$. 
A similar computation as in the previous case yields
$$ \partial_{\theta\theta} S_\ell + \mu_\ell^2 S_\ell 
= \sum_{j=1}^m \sum_{\mu_{k_1} + \cdots + \mu_{k_j} = \mu_m} a_j P_{k_1} \cdots P_{k_j}.$$
By Proposition \ref{prop-polynomial-coefficients}(4b), we obtain \eqref{eq-R_ell = P_ell}. This finishes the proof by induction.
\end{proof}

By Proposition \ref{prop-formal-expansion-multi-index}, we can write \eqref{eq-formal-expansion-reparam} as
\begin{align*}
\sum_{\alpha \in \varLambda} P_\alpha(\theta) e^{- \langle\alpha,\rho\rangle t}.
\end{align*}
We now make a remark. 
Let $\mathcal{X} = \{X_i\}_{i\geq 1}$ be a sequence of spherical harmonics 
with $X_i\in\mathscr{E}_i$. 
In the parametrization $\{P_\ell(\mathcal{X})\}_{\ell \geq 1}$ given by $\ell \geq 1$, 
$P_\ell(\mathcal{X})$ is described by Proposition \ref{prop-polynomial-coefficients}. 
Roughly speaking, if $\mu_\ell$ is not an integer, then $P_\ell$ is uniquely determined 
by $P_1,\cdots,P_{\ell-1}$ and $a_0/\tau_0^2$, and if $\mu_\ell = n$ is an integer, 
then either $P_\ell = X_n$ or $P_\ell = X_n + Q_\ell$ for some $Q_\ell$ uniquely determined 
by $P_1,\cdots,P_{\ell-1}$ and $a_0/\tau_0^2$, where $X_n$ is one of the prescribed spherical harmonics.  
On the other hand, for the parametrization $\{P_\alpha(\mathcal{X})\}_{\alpha \in \varLambda}$ 
given by $\alpha \in \varLambda$, we require $P_{e_0} = a_0/\rho_0^2$ and $P_{e_i} = X_i$ for any $i \geq 1$, 
and, for $|\alpha| \geq 2$, either $P_\alpha=0$ or $P_\alpha$ 
is the unique solution to \eqref{eq-recursive-equation-multi-index} such that $P_\alpha$ is perpendicular to 
$\mathrm{Ker}\, (\partial_{\theta\theta} + \langle\alpha,\rho\rangle^2)$ in $L^2(\mathbb{S}^1)$. 
Hence, in the case where $\langle\alpha,\rho\rangle = n$ for some $n \geq 1$ with $\alpha_0 \ge 1$ and $|\alpha|\ge 2$, 
although \eqref{eq-recursive-equation-multi-index} does not have a unique solution, 
$P_\alpha (\mathcal{X})$ is unique under the additional requirement that $P_\alpha(\mathcal{X})$ 
is perpendicular to $\mathscr{E}_n$ in $L^2(\mathbb{S}^1)$.

\section{Convergence of the Expansions}
\label{sec-Convergence-of-the-Expansions}

In this section, we construct solutions to \eqref{eq-linearized-equation} 
if $f$ is an analytic function on $(-\epsilon_0, \epsilon_0)$, for some $\epsilon_0 >0$. 
More precisely, we prove that the formal expansion
\begin{align}
\label{eq-formal-expansion-multi-index-1}
\sum_{\alpha \in \varLambda} P_\alpha (\mathcal{X})(\theta) e^{-\langle\alpha,\rho\rangle t},
\end{align}
determined by a sequence $\mathcal{X} =\{X_i\}_{i \geq 1}$ of spherical harmonics with $X_i \in \mathscr E_i$,
converges to an exponential decay solution $v$ of \eqref{eq-linearized-equation},
and the asymptotic behavior of $v$ is determined by $\mathcal{X}$.

The proof of the convergence of the series \eqref{eq-formal-expansion-multi-index-1} 
is based on the method of majorant. The next lemma provides an upper bound.

\begin{lemma}
\label{lemma-majorant-series}
Let $A_0$, $A$, $B$, and $M$ be given nonnegative constants. Define $A_{e_0} = A_0$, $A_{e_i}  = A^i$ for $i \geq 1$, 
and $A_\alpha = A_\alpha(A_0,A,B,M)$ inductively for $\alpha = \{\alpha_i\}_{i\ge 0} \in \varLambda$ with $|\alpha|\geq 2$ 
by $A_\alpha = 0$ if $\alpha_0 = 0$ and
\begin{align}
\label{eq-def-A_alpha}
A_\alpha = B \sum_{j = 1}^{|\alpha|-1}M^j \sum_{\beta_1 + \cdots + \beta_j = \alpha - e_0} A_{\beta_1} \cdots A_{\beta_j},
\end{align}
if $\alpha_0 \ge 1$.
Then, for any $k,\ell \ge 0$ with $k+\ell \ge 1$,
\begin{align}
\label{eq-majorant-series}
\sum_{\substack{\alpha = (\alpha_0,\tilde \alpha)\\\alpha_0 = k,\,\langle\tilde\alpha,\N\rangle = \ell}}A_\alpha \le D^{k+\ell},
\end{align}
for some positive constant $D$ depending only on $A_0$, $A$, $B$, and $M$.
\end{lemma}

\begin{proof}
For $\alpha \in \varLambda$, we write
$$\alpha = (k,\tilde\alpha)\quad\text{and}\quad A_\alpha = A_{(k,\tilde\alpha)},$$
and define, for any $k,\ell \ge 0$ with $k+\ell \ge 1$, 
$$B_{k,\ell}  = \sum_{\langle \tilde\alpha,\N\rangle = \ell} A_{(k,\tilde\alpha)}.$$
Then, for any $(k,\ell)$ with $k\ge 1$ and $k+\ell \ge 2$,
\begin{align*}
B_{k,\ell} = \sum_{\langle \tilde\alpha,\N\rangle = \ell} A_{(k,\tilde\alpha)} 
= \sum_{\langle \tilde\alpha,\N\rangle = \ell}
B \sum_{j=1}^{k+|\tilde\alpha|-1} M^j 
\sum_{(k_1,\tilde\beta_1) + \cdots + (k_j,\tilde\beta_j) = (k-1,\tilde\alpha)} 
A_{(k_1,\tilde\beta_1)}\cdots A_{(k_j,\tilde\beta_j)}.
\end{align*}
Here, we sum over all $(k_1,\tilde \beta_1), \cdots, (k_j,\tilde\beta_j)$ such that 
$$k_1 +\cdots + k_j = k-1\quad\text{and}\quad\tilde \beta_1 + \cdots + \tilde\beta_j = \tilde\alpha,$$ 
for some $\tilde\alpha$ with $\langle\tilde\alpha,\N\rangle = \ell$.
In other words, we sum over all possible $(k_1,\tilde\beta_1),\cdots,(k_j,\tilde \beta_j)$ such that
$$k_1 +\cdots + k_j = k-1\quad\text{and}\quad\langle\tilde \beta_1,\N\rangle + \cdots + \langle\tilde\beta_j,\N\rangle = \ell.$$ 
Note that 
$$1 \le j \le \sum_{s=1}^j |(k_s,\tilde\beta_s)| \le \sum_{s=1}^j \big(k_s + \langle\tilde\beta_s,\N\rangle \big) = k+\ell-1.$$
By switching the order of summation, we have
$$\sum_{\langle \tilde\alpha,\N\rangle = \ell}
\sum_{j=1}^{k+|\tilde\alpha|-1} 
\sum_{(k_1,\tilde\beta_1) + \cdots + (k_j,\tilde\beta_j) = (k-1,\tilde\alpha)} 
=
\sum_{j=1}^{k+\ell-1} \sum_{k_1 + \cdots +k_j = k-1}\sum_{\substack{\tilde\beta_1,\cdots,\tilde\beta_j\\
\langle\tilde\beta_1,\N\rangle + \cdots + \langle\tilde\beta_j,\N\rangle = \ell}}.
$$
Then, by a simple rearrangement of the last summation, we have
$$\sum_{\langle \tilde\alpha,\N\rangle = \ell}
\sum_{j=1}^{k+|\tilde\alpha|-1} 
\sum_{(k_1,\tilde\beta_1) + \cdots + (k_j,\tilde\beta_j) = (k-1,\tilde\alpha)} 
=
\sum_{j=1}^{k+\ell-1} \sum_{(k_1,\ell_1) + \cdots +(k_j,\ell_j) = (k-1,\ell)}\sum_{\substack{\tilde\beta_1,\cdots,\tilde\beta_j\\
\langle\tilde\beta_1,\N\rangle=\ell_1,  \cdots,  \langle\tilde\beta_j,\N\rangle = \ell_j}}.$$
Therefore, we obtain, for any $k\ge 1$ and $k+\ell\ge 2$,
\begin{align*}
&\sum_{\langle \tilde\alpha,\N\rangle = \ell}
B \sum_{j=1}^{k+|\tilde\alpha|-1} M^j 
\sum_{(k_1,\tilde\beta_1) + \cdots + (k_j,\tilde\beta_j) = (k-1,\tilde\alpha)} 
A_{(k_1,\tilde\beta_1)}\cdots A_{(k_j,\tilde\beta_j)}\\
&\quad=B\sum_{j=1}^{k+\ell-1}
M^j \sum_{(k_1,\ell_1) +\cdots + (k_j,\ell_j) = (k-1,\ell)} 
\big( \sum_{\tilde\beta_1,\langle\tilde\beta_1,\N\rangle = \ell_1} 
A_{(k_1,\tilde\beta_1)} \big) \cdots 
\big( \sum_{\tilde\beta_j,\langle\tilde\beta_j,\N\rangle = \ell_j} A_{(k_j,\tilde\beta_j)} \big)\\
&\quad= B\sum_{j=1}^{k+\ell-1} 
M^j \sum_{(k_1,\ell_1) +\cdots + (k_j,\ell_j) = (k-1,\ell)} 
B_{k_1,\ell_1}\cdots B_{k_j,\ell_j}.
\end{align*}
In summary, we have, for any $(k,\ell)$ with $k\ge 1$ and $k+\ell \ge 2$,
\begin{align}
\label{eq-recursive-equation-B_k,l}
B_{k,\ell} = B\sum_{j=1}^{k+\ell-1} 
M^j \sum_{(k_1,\ell_1) +\cdots + (k_j,\ell_j) = (k-1,\ell)} 
B_{k_1,\ell_1}\cdots B_{k_j,\ell_j}.
\end{align}
Now, consider the formal expansion 
$$\varphi (s,t)= \sum_{(k,\ell)\ne 0} B_{k,\ell} s^kt^\ell\quad\text{near }0.$$
Note that right-hand side of \eqref{eq-recursive-equation-B_k,l} is the $(k,\ell)$-th coefficient of the expansion of
$$sB \sum_{j=1}^\infty M^j \varphi^j = sB\frac{M \varphi}{1 - M\varphi}. $$
It is obvious that $k\ge 1$ and $k+\ell \ge 2$ are violated if and only if $k = 0$ or $k = 1$ and $\ell = 0$.
Then, formally, $\varphi=\varphi(s,t)$ satisfies the algebraic equation
\begin{align}\label{eq-equation-phi}\varphi - \sum_{\ell=1}^\infty A^\ell t^\ell - A_0 s = sB\frac{M \varphi}{1 - M\varphi},
\end{align}
where we used $B_{0,\ell} = A_{e_\ell} = A^\ell$ for $\ell \ge 1$, and $B_{1,0} = A_{e_0} = A_0$.
Since $\varphi(0,0) = 0$, the algebraic equation \eqref{eq-equation-phi}
has a unique analytic solution $\varphi$ near $0$. 
In particular, 
there exists $D  = D(A_0,A, B,M) > 0$ such that \eqref{eq-majorant-series} holds.
\end{proof}

The estimate \eqref{eq-majorant-series} is referred to as an {\it analyticity-type estimate}.
We point out that, for any $0 < \eta < 1$,
$$\sum_{(k,\ell)\ne (0,0)}\eta^{k+\ell} = \frac{1}{(1-\eta)^2}-1.$$

In the study of nonlinear equations, it is convenient to adopt certain Sobolev norms. 
We now introduce an inner product involving derivatives. 
Fix a positive integer $2_*$. 
For any $u,v \in H^{2_*}(\mathbb{S}^1)$, 
consider the standard $H^{2_*}$-inner product given by
\begin{align*}
\langle u, v\rangle_{*} 
= \sum_{i=0}^{2_*} \langle \partial_\theta^i u, \partial_\theta^i v\rangle_{L^2(\mathbb{S}^1)},
\end{align*}
where $\partial_{\theta}^0 \equiv \mathrm{Id}$, and write 
$$\|u\|_{*} = \langle u,u\rangle_{*}^\frac{1}{2}.$$ 
We now make some remarks. 
First, we have the Banach algebra property by the $1$-dimensional Sobolev inequality. 
Specifically, there exists a constant $C > 1$ such that
\begin{align}
\label{eq-Banach-algebra-property-with-constant}
\|u v\|_{*} \leq C \|u\|_{*}\|v\|_{*}, 
\end{align}
for any $u,v \in H^{2_*}(\mathbb{S}^1)$. 
To simplify notations, we will drop the constant in \eqref{eq-Banach-algebra-property-with-constant}. 
In other words, we re-define $\langle\cdot,\cdot\rangle_{*}$ by
\begin{align}
\label{eq-H^m-inner-product-with-constant}
\langle u,v \rangle_{*} 
=  C^2 \sum_{i=0}^{2_*} \langle \partial_\theta^i u, \partial_\theta^i v\rangle_{L^2(\mathbb{S}^1)},
\end{align}
for any $u,v \in H^{2_*}(\mathbb{S}^1)$. Then, \eqref{eq-Banach-algebra-property-with-constant} reduces to
\begin{align}
\label{eq-Banach-algebra-property-without-constant}
\|u v\|_{*} \leq \|u\|_{*}\|v\|_{*},
\end{align}
for any $u,v \in H^{2_*}(\mathbb{S}^1)$. 
Second, by construction, $\partial_{\theta\theta}$ is self-adjoint in $H^{2_*}(\mathbb{S}^1)$ in the sense that
$$\langle \partial_{\theta\theta} u, v\rangle_{*} = \langle  u,  \partial_{\theta\theta} v\rangle_{*},$$
for any $u,v \in C^\infty(\mathbb{S}^1)$.


Now, we are ready to prove that the formal series \eqref{eq-formal-expansion-multi-index-1} 
converges to a smooth solution $v$ of \eqref{eq-linearized-equation} with prescribed asymptotic behavior. 

\begin{theorem}
\label{thrm-main-theorem-cylindrical-coords}
Assume that $\tau_0>0$ is a constant and $f$ is an analytic function on $(-\epsilon_0,\epsilon_0)$ given by 
$$f(s) = \sum_{j=0}^\infty a_j s^j,$$
with $|a_j| \leq M^j$, $j \geq 1$, for some $M>0$.  
Let $\mathcal{X} =\{X_i\}_{i \geq 1} $ 
be a sequence of spherical harmonics with $X_i \in \mathscr{E}_i$ 
and $\{P_\alpha(\mathcal{X})\}_{\alpha \in \varLambda}$ be 
the sequence of smooth functions on $\mathbb S^1$ determined by $\mathcal{X}$
as in Proposition \ref{prop-formal-expansion-multi-index}. 
Assume that, for any $i \geq 1$,
\begin{align}
\label{eq-exponential-growth-condition-expansion-1}
\|X_i\|_{L^2(\mathbb{S}^1)} \leq \tilde A^i,
\end{align}
for some constant $\tilde A>0$.
Then, there exists a positive constant 
$t_0$, depending only on $a_0$, $\tau_0$, $\tilde A$, and $M$, 
such that the series \eqref{eq-formal-expansion-multi-index-1}
converges in the $C^m$-sense to a smooth function $v$
on $[t_0,\infty)\times\mathbb{S}^1$, for any $m \ge 0$.
Moreover, $v$ is a solution to \eqref{eq-linearized-equation}
on $[t_0,\infty)\times \mathbb S^1$ with an exponential decay and 
its asymptotic behavior is determined by $\mathcal{X} =\{X_i\}_{i \geq 1}$.
\end{theorem}

\begin{proof} The proof consists of three steps. 

{\it Step 1. We prove that the series \eqref{eq-formal-expansion-multi-index-1}
converges uniformly and absolutely in $[t_0,\infty)\times \mathbb{S}^1$,
for some $t_0 \ge 1$ depending only on $a_0$, $\tau_0$, $\tilde A$, and $M$.}
First, note that, for $X \in \mathscr{E}_i$, 
$$\|X\|_{*} = C(1 + \lambda_i + \cdots + \lambda_i^{2_*})^{\frac{1}{2}} \|X\|_{L^2(\mathbb{S}^1)},$$
where $C$ is the constant in \eqref{eq-H^m-inner-product-with-constant}.
Hence, 
$$\|X\|_{*} \leq C(2_*+1)^{\frac{1}{2}} \lambda_i^\frac{2_*}{2} \|X\|_{L^2(\mathbb{S}^1)} 
\leq \tilde C^i \|X\|_{L^2(\mathbb{S}^1)},$$
for some constant $\tilde C > 0$, independent of $i$. 
By \eqref{eq-exponential-growth-condition-expansion-1}, 
there exists $ A> 0$, depending only on $\tilde A$, such that, for any $i \geq 1$,
\begin{align}
\label{eq-exponential-growth-condition-H^m-norm-expansion}
\|X_i\|_{*} \leq  A^i.
\end{align}

For simplicity, we write $P_\alpha = P_\alpha(\mathcal{X})$. 
We claim that, for any $\alpha \in \varLambda$ with $|\alpha| \ge 2$ and $\alpha_0 \ge 1$,
\begin{align}
\label{eq-recursive-estimate-LHS<=RHS}
\|P_\alpha\|_{*} \leq \rho_0^{-2} \sum_{j = 1}^{|\alpha|-1}M^j 
\sum_{\beta_1 + \cdots + \beta_j = \alpha - e_0}  \|P_{\beta_1}\|_* \cdots \|P_{\beta_j}\|_*.
\end{align}
Set
\begin{align}
\label{eq-def-R_alpha}
R_\alpha = \sum_{j = 1}^{|\alpha|-1} \sum_{\beta_1 + \cdots + \beta_j = \alpha - e_0} a_j P_{\beta_1} \cdots P_{\beta_j}.
\end{align}
By \eqref{eq-recursive-equation-multi-index}, we have
\begin{align}
\label{eq-reduced-equation-P_alpha}
\partial_{\theta\theta} P_\alpha + \langle\alpha,\rho\rangle^2 P_\alpha = R_\alpha.
\end{align}
By Lemma \ref{lemma-decomposition-of-spherical-harmonics} and Proposition \ref{prop-formal-expansion-multi-index},
the order of $R_\alpha$ is at most $\langle\alpha,\Z_+\rangle$. 
In particular, we can write
$$R_\alpha = \sum_{i=0}^{\langle\alpha,\Z_+\rangle} R_{\alpha,i},$$
for some $R_{\alpha,i} \in \mathscr{E}_i$, $i = 0,\cdots, \langle\alpha,\Z_+\rangle$.
Note that
\begin{align*}
\langle\alpha,\rho\rangle^2 - \langle\alpha,\Z_+\rangle^2 
= \alpha_0^2\rho_0^2 + 2 \alpha_0\rho_0\langle\tilde\alpha,\N\rangle \ge \rho_0^2. 
\end{align*}
By \eqref{eq-reduced-equation-P_alpha} and $\langle\alpha,\Z_+\rangle< \langle\alpha,\rho\rangle$, we get
$$P_\alpha = \sum_{i=0}^{\langle\alpha,\Z_+\rangle} (\langle\alpha,\rho\rangle^2 - i^2)^{-1} R_{\alpha,i}.$$
Then, $\|P_\alpha\|_* \le \rho_0^{-2} \|R_\alpha\|_*$, 
and hence \eqref{eq-recursive-estimate-LHS<=RHS} follows from \eqref{eq-Banach-algebra-property-without-constant} and \eqref{eq-def-R_alpha}.

Now, recall that $\rho_0 = \tau_0>0$ and $\rho_i= i$ for any $i \ge 1$. 
Let $A_0 = \||a_0|/\rho_0^2\|_{*}$, 
$A$ be the constant in \eqref{eq-exponential-growth-condition-H^m-norm-expansion},
$B = \max \{\rho_0^{-2}, 1\}$,
and $M$ be the constant in \eqref{eq-recursive-estimate-LHS<=RHS},
and let $A_\alpha$ and $D$ be the constants in Lemma \ref{lemma-majorant-series}
depending only on $A_0$, $A$, $B$, and $M$. 
Then, by \eqref{eq-def-A_alpha}, \eqref{eq-exponential-growth-condition-H^m-norm-expansion}, 
\eqref{eq-recursive-estimate-LHS<=RHS}, and a simple induction argument, 
we obtain, for any $\alpha \in \varLambda$,
\begin{align}
\label{eq-P_alpha<=A_alpha}
\|P_\alpha\|_* \le A_\alpha.
\end{align}
By \eqref{eq-majorant-series}, we obtain,
for any $k,\ell \ge 0$ with $k+\ell \ge 1$, 
$$\sum_{\substack{\alpha = (\alpha_0,\tilde \alpha)\\\alpha_0 = k,\,\langle\tilde\alpha,\N\rangle = \ell}} \|P_\alpha\|_* e^{-\langle\alpha,\rho\rangle t} \le (D e^{-\rho_0 t})^{k} (D e^{-t})^{\ell},$$
and hence, by the Sobolev embedding, 
for any $(t,\theta) \in [0,\infty)\times \mathbb S^{n-1}$,
$$\sum_{\substack{\alpha = (\alpha_0,\tilde \alpha)\\\alpha_0 = k,\,\langle\tilde\alpha,\N\rangle = \ell}} |P_\alpha(\theta)| e^{-\langle\alpha,\rho\rangle t} \le C_0(D e^{-\rho_0 t})^{k} (D e^{-t})^{\ell},$$
for some positive constant $C_0 $.
We now take $t_0\ge 1$ such that, for any $t \ge t_0$, 
$$D e^{-\rho_0 t} \le 1/2\quad\text{and}\quad D e^{-t} \le 1/2.$$
Hence, the series \eqref{eq-formal-expansion-multi-index-1} converges
uniformly and absolutely on $[t_0,\infty)\times \mathbb S^{1}$.

{\it Step 2. We improve the convergence to the $C^m$-sense, for any $m \ge 1$.}
We will use the following estimate:
{\it Let $P \in C^\infty(\mathbb S^1)$ be of order at most $N$, 
then, for any $m \ge 0$,
\begin{align}
\label{eq-estimate-C^m-polynomial-order-N}
\|\partial_{\theta}^m P\|_* \le N^m \|P\|_*.
\end{align}}
We point out that \eqref{eq-estimate-C^m-polynomial-order-N} follows from the fact that any $X \in \mathscr E_i$ can be written as a linear combination of $\cos (i\theta)$ and $\sin(i\theta)$.
Recall that, by Proposition \ref{prop-formal-expansion-multi-index}, $P_\alpha$ has an order at most $\langle\alpha,\Z_+\rangle$.
By \eqref{eq-P_alpha<=A_alpha}, \eqref{eq-estimate-C^m-polynomial-order-N}, and the Sobolev embedding,
for any $\alpha \in \varLambda$ and nonnegative integers $m_1,m_2$,
$$|\partial_\theta^{m_1} \partial_t^{m_2} (P_\alpha(\theta) e^{-\langle\alpha,\rho\rangle t})| \le C_0 \langle\alpha,\Z_+\rangle^{m_1} \langle\alpha,\rho\rangle^{m_2} A_\alpha,$$
for some positive constant $C_0$.
Since $\langle\alpha,\Z_+\rangle \le \langle\alpha,\rho\rangle$, 
we get, for any $m \ge 1$,
$$|\nabla_{(t,\theta)}^{m}  (P_\alpha(\theta) e^{-\langle\alpha,\rho\rangle t})| \le C_0 (m+1)  \langle\alpha,\rho\rangle^{m} A_\alpha.$$
Therefore, for any $k,\ell \ge 0$ with $k+\ell \ge 1$,
$$\sum_{\substack{\alpha = (\alpha_0,\tilde \alpha)\\\alpha_0 = k,\,\langle\tilde\alpha,\N\rangle = \ell}}|\nabla_{(t,\theta)}^{m}  (P_\alpha(\theta) e^{-\langle\alpha,\rho\rangle t})| \le C_0 (m+1)  (\rho_0 k+\ell)^m (D e^{-\rho_0 t})^k (D e^{-t})^\ell.$$
We point out that, for any nonnegative integer $p,q$,
$$\sum_{(k,\ell)\ne (0,0)} k^p \ell^q \frac{1}{2^{k+\ell}} < \infty.$$
By our choice of $t_0$ in Step 1, the series 
\begin{align*}
\sum_{\alpha \in \varLambda} \nabla^m_{(t, \theta)}\big(P_\alpha(\theta) e^{-\langle\alpha,\rho\rangle t}\big)\end{align*}
converges uniformly and absolutely on $[t_{0}, \infty) \times \mathbb{S}^{n-1}$.
This holds for any $m\ge 1$.
Therefore, the series
$$\sum_{\alpha \in \varLambda} P_\alpha e^{-\langle\alpha,\rho\rangle t}$$
converges in the $C^m$-sense to a smooth function $v$ on $[t_0,\infty)\times\mathbb S^1$,
for any $m \ge 0$.
Moreover, by \eqref{eq-recursive-equation-multi-index} and a simple rearrangement,
it is straightforward to check that $v$ is a solution to \eqref{eq-linearized-equation}.

{\it Step 3. We prove  that the asymptotics of $v$ is determined by $\mathcal{X}$.}
As in Proposition \ref{prop-formal-expansion-multi-index}, we write
\begin{align}
\label{eq-formal-expansion-last}
v (t,\theta) = \sum_{\ell=1}^\infty P_\ell(\theta)e^{-\mu_\ell t},
\end{align}
where
$$P_\ell = \sum_{\alpha,\langle\alpha,\rho\rangle = \mu_\ell} P_\alpha.$$
The series in \eqref{eq-formal-expansion-last} converges in the $C^m$-sense on $[t_0,\infty)\times\mathbb S^1$ for any $m \ge 0$.
Fix $\ell \ge 1$. We have
$$v - \sum_{j=1}^\ell P_j(\theta) e^{-\mu_j t} 
= \sum_{\alpha,\langle\alpha,\rho\rangle \ge \mu_{\ell+1}} P_\alpha(\theta) e^{-\langle\alpha,\rho\rangle t}.$$
Choose $K>0$ such that $\mu_{\ell+2} - \mu_{\ell+1} \ge \mu_{\ell+2}/K$. 
Then, for any $\langle\alpha,\rho\rangle \ge \mu_{\ell+2}$,
$$\langle\alpha,\rho\rangle - \mu_{\ell+1} \ge \langle \alpha,\rho\rangle/K.$$
Hence, for any $t \ge t_{0,K}$,
\begin{align*}
\Big|\sum_{\alpha,\langle\alpha,\rho\rangle \ge \mu_{\ell+1}} P_\alpha(\theta) e^{-\langle\alpha,\rho\rangle t} \Big| 
&\le \big|P_{\ell+1}(\theta) e^{-\mu_{\ell+1} t} \big| + e^{-\mu_{\ell+1} t}
\Big|\sum_{\alpha,\langle\alpha,\rho\rangle\ge \mu_{\ell+2}} P_\alpha(\theta) e^{-(\langle\alpha,\rho\rangle - \mu_{\ell+1}) t}  \Big|\\
& \le \big|P_{\ell+1}(\theta) e^{-\mu_{\ell+1} t} \big| + e^{-\mu_{\ell+1} t}
\Big|\sum_{\alpha,\langle\alpha,\rho\rangle\ge \mu_{\ell+2}} P_\alpha(\theta) e^{-\langle\alpha,\rho\rangle t/K}  \Big|. 
\end{align*}
We now choose a positive constant $t_{0,K}$ such that, for any $t \ge t_{0,K}$,
$$ De^{-\rho_0t/K}\le 1/2\quad\text{and}\quad  De^{t/K}\le 1/2.$$
As in the proof of Step 2, up to a multiplicative constant, the last sum is controlled by
$$\sum_{(k,\ell)\ne (0,0)} \frac{1}{2^{k+\ell}} = 3.$$
Therefore, we obtain, for any $\ell \ge 1$,
$$v - \sum_{j=1}^\ell P_j(\theta) e^{-\mu_j t} = O(e^{-\mu_{\ell+1} t}).$$
In other words, the asymptotics of $v$ is determined by $\mathcal{X}$.
\end{proof}

The solution $v$ in Theorem \ref{thrm-main-theorem-cylindrical-coords} 
determined by $\mathcal{X}$ is unique by the unique continuation for elliptic equations with singular coefficients.
To be specific, assume that $v_1$ and $v_2$ are two solutions to \eqref{eq-linearized-equation} 
and their asymptotic behaviors are determined by
the same sequence of spherical harmonics $\{X_i\}_{i\ge 1}$ with $X_i \in \mathscr{E}_i$.
By Theorem \ref{thrm-asymptotic-expansion}, $w = v_1 - v_2 = O(e^{-\mu t})$ as $t\rightarrow\infty$ for any $\mu>0$.
In the Euclidean coordinates, $w = O(|x|^{\mu})$ as $|x|\rightarrow 0$ for any $\mu > 0$, and
$$\Delta w = |x|^{\tau_0-2} c(x) w,$$
for some constant $\tau_0 > 0$ and bounded function $c$ near the origin. 
By the unique continuation, we conclude $w = 0$, i.e., $v_1 = v_2$.
Refer to \cite{Hormander}, \cite{JerisonKenig}, and \cite{Meshkov} 
for proofs by Carleman-type estimates, 
or to \cite{GarofaloLin1} and \cite{GarofaloLin2} by monotonicity formula.


Theorem \ref{thrm-correspondence-harmonic-to-solutions-Intro} 
follows from Theorem \ref{thrm-main-theorem-cylindrical-coords}.

\end{document}